\documentclass[11pt]{article}

\usepackage{paralist, enumerate, amsmath, amsthm, amsfonts, amssymb, color, mathrsfs,comment}

\usepackage{tikz}
\usepackage{tkz-graph}
\usetikzlibrary{arrows,%
                shapes,positioning}

\definecolor{DarkBlue}{rgb}{0,0,0.8} 
\definecolor{DarkGreen}{rgb}{0,0.5,0.0} 
\definecolor{DarkRed}{rgb}{0.9,0.0,0.0} 

\usepackage[T1]{fontenc}
\usepackage[latin1]{inputenc}
\usepackage{geometry}
\usepackage{framed}
\usepackage{hyperref}
\usepackage{graphicx}

\numberwithin{equation}{section}
\newtheorem{theorem}[equation]{Theorem}
\newtheorem{lemma}[equation]{Lemma}
\newtheorem{cor}[equation]{Corollary}
\newtheorem{prop}[equation]{Proposition}

\theoremstyle{definition}

\newtheorem{open}[equation]{Question}

\newcommand{\ZZ}{\mathbf{Z}}
\newcommand{\NN}{\mathbf{N}}
\newcommand{\RR}{\mathbf{R}}

\newcommand{\N}{N}
\newcommand{\po}[2]{\mathfrak{po}^{#1|#2}}
\newcommand{\on}{\operatorname}

\newcommand{\Img}{\on{Im}}

\newcommand{\Q}{\overline{q}}
\newcommand{\w}{\on{weight}}

\newcommand{\smon}{\mathbf{SMon}}
\newcommand{\clif}{\on{clif}}
\newcommand{\cl}{\mathbf{Cl}}
\newcommand{\inc}{\on{inc}}
\newcommand{\cut}[4]{#1 = #2 \amalg_{#4} #3}
\newcommand{\piece}[3]{#1_{#3}}
\newcommand{\wt}{\on{wt}}



\title{Adinkras for Mathematicians}

\author{Yan X Zhang \\ Massachusetts Institute of Technology}
\date{\today}
\begin{document}

\maketitle

\begin{abstract}
\emph{Adinkras} are graphical tools created for the study of representations in supersymmetry. Besides having
inherent interest for physicists, adinkras offer many easy-to-state and accessible mathematical problems of algebraic, combinatorial, and computational nature. We use a more mathematically natural language to survey these topics, suggest new definitions, and present original results.
\end{abstract}

\section{Introduction}

In a series of papers starting with \cite{d2l:first}, different subsets of the ``DFGHILM collaboration'' (Doran, Faux, Gates, H\"{u}bsch, Iga, Landweber, Miller) have built and extended the machinery of \emph{adinkras}. Following the spirit of Feyman diagrams, adinkras are combinatorial objects that encode information about the representation theory of supersymmetry algebras. Adinkras have many intricate links with other fields such as graph theory, Clifford theory, and coding theory. Each of these connections provide many problems that can be compactly communicated to a (non-specialist) mathematician.
 
This paper is a humble attempt to bridge the language gap and generate communication. We redevelop the foundations in a self-contained manner in Sections~\ref{sec:adinkras} and \ref{sec:topologies}, using different definitions and constructions that we consider to be more mathematically natural for our purposes. Using our new setup, we prove some original results and make new interpretations in Sections~\ref{sec:dashing} and \ref{sec:ranking}. We wish that these purely combinatorial discussions will equip the readers with a mental model that allows them to appreciate (or to solve!) the original representation-theoretic problems in the physics literature. We return to these problems in Section~\ref{sec:together} and reconsider some of the foundational questions of the theory.

\section{Definitions}
\label{sec:adinkras}

We assume basic notions of graphs. For a graph $G$, we use $E(G)$ to denote the edges of $G$ and $V(G)$ to denote the vertices of $G$. We deviate from the original literature\footnote{In the existing literature, the notion of posets is created from scratch without ever being referred to as such. We make a more compact presentation, using preexisting language when possible. Furthermore, the existing literature defines adinkras first and then defines topologies and chromotopologies as coming from adinkras (i.e. what we call \emph{adinkraizable}). We modualarize the data and take a ``ground-up'' approach instead.}, but we believe our approach yields slightly cleaner and more general mathematics while getting to the same destinations.

We assume most basic notions of posets (there are many references, including \cite{stanley:ec1}). In this paper, we think of each Hasse diagram for a poset as a directed graph, with $x \to y$ an edge if $y$ covers $x$. Thus it makes sense to call the maximal elements (i.e. those $x$ with no $y > x$) \emph{sinks} and the minimal elements \emph{sources}. 

A \emph{ranked} poset\footnote{This is also often called a \emph{graded} poset, though there are similar but subtly different uses of that name. For this paper, we use \emph{ranked} to avoid ambiguity.} is a poset $A$ equipped with a rank function $h\colon A \to \ZZ$ such that for all $x$ covering $y$ we have $h(x) = h(y) + 1$. There is a unique rank function $h_0$ among these such that $0$ is the lowest value in the range of $h_0$, so it makes sense to define the \emph{rank} of an element $v$ as $h_0(v)$. The largest element in the range of $h_0$ is then the length of the longest chain in $A$; we call it the \emph{height} of $A$.

\subsection{Topologies, Chromotopologies, and Adinkras}
\label{sec:definitions}
%A possible complication with the literature of adinkras is the variable nature of the word ``adinkra'' itself, a consequence of the existence of associated properties of an adinkra that may or may not be present in each discussion and must be deduced through context. In this section, we try to make everything precise and offer enough flexibility to talk about adinkras in all contexts. 

An $n$-dimensional \emph{adinkra topology}, or \emph{topology} for short, is a finite connected simple graph $A$ such that $A$ is bipartite and \emph{$n$-regular} (every vertex has exactly $n$ incident edges). We call the two sets in the bipartition of $V(A)$ \emph{bosons} and \emph{fermions}, though the actual choice is mostly arbitrary and we do not consider it part of the data.

A \emph{chromotopology} of dimension $n$ is a topology $A$ such that the following holds.
\begin{itemize}
\item The elements of $E(A)$ are colored by $n$ colors, which are elements of the set $[n] = \{1,2,\ldots, n\}$ unless denoted otherwise, such that every vertex is incident to exactly one edge of each color. 
\item For any distinct $i$ and $j$, the edges in $E(A)$ with colors $i$ and $j$ form a disjoint union of $4$-cycles.  %The $q_i$ commute. Equivalently, they generate a free $\ZZ_2^n$ action on $V(A)$. The combinatorial manifestation of this requirement is that for any distinct $i$ and $j$, the edges in $E(A)$ with colors $i$ and $j$ form a disjoint union of $4$-cycles. 
\end{itemize}

There are further properties we can put on a chromotopology:
\begin{enumerate}
\item {\bf ranked}: we give $A$ additional structure of a ranked poset\footnote{In the existing literature, such as \cite{d2l:graph-theoretic}, these adinkras are constructed as part of the data of the graph, using directed edges to encode cover relations (the graph with the poset structure is called \emph{engineerable}), developing the notion of posets from scratch without ever referring to them as such. In this paper, we separate the poset structure from the graph and refer to the well-developed notion of \emph{ranked poset}.}, with rank function $h_A$ (though we will usually just write $h$). By this, we mean that we identify $A$ with the Hasse diagram of some ranked poset and assign the corresponding ranks to $V(A)$. In this paper, such as in Figure~\ref{fig:3cube}, we will usually represent ranks via vertical placement, with higher $h$ corresponding to being higher on the page.
\item {\bf dashed}: we add an \emph{odd dashing} $A$, which is a map $d\colon E(A) \to \ZZ_2$ such that the sum of $d(e)$ as $e$ runs over each $2$-colored $4$-cycle (that is, a $4$-cycle of edges that use a total of $2$ colors) is $1 \in \ZZ_2$. Visually, we can think of the odd dashing as making each edge of $A$ either \emph{dashed} or \emph{solid}, such that every $2$-colored $4$-cycle contains an odd number of dashes. We will slightly abuse notation and write $d(v,w)$ to mean $d((v,w))$, where $(v,w)$ is an edge from $v$ to $w$.
\end{enumerate}
An \emph{adinkra} is a chromotopology with both of these properties: i.e. a dashed ranked chromotopology. We call a topology or chromotopology that can be made into an adinkra \emph{adinkraizable}.

\begin{figure}[htb]
\begin{center}

\begin{tikzpicture}[scale=0.15]
%\SetVertexNormal
\SetUpEdge[labelstyle={draw}]
\Vertex[x=0,y=0]{000}
\Vertex[x=0,y=10]{010}
\Vertex[x=0,y=20]{101}
\Vertex[x=0,y=30]{111}
\Vertex[x=-10,y=10]{100}
\Vertex[x=10,y=10]{001}
\Vertex[x=-10,y=20]{110}
\Vertex[x=10,y=20]{011}
\draw (20, 0) node []{$h = 0$};
\draw (20, 10) node []{$h = 1$};
\draw (20, 20) node []{$h = 2$};
\draw (20, 30) node []{$h = 3$};
\Edge[color=red](100)(101)
\Edge[color=red](000)(001)
\Edge[color=red](010)(011)
\Edge[color=red](110)(111)
\Edge[color=green](000)(100)
\Edge[color=green, style=dashed](001)(101)
\Edge[color=green](010)(110)
\Edge[color=green, style=dashed](011)(111)
\Edge[color=blue](000)(010)
\Edge[color=blue, style=dashed](001)(011)
\Edge[color=blue, style=dashed](100)(110)
\Edge[color=blue](101)(111)
\end{tikzpicture}

\caption{An adinkra with vertices labeled by $3$ bits. We can take the bosons to be either $\{000,011,101,110\}$ or $\{001,010,100,111\}$ and take the fermions to be the other set. \label{fig:3cube}}
\end{center}
\end{figure}
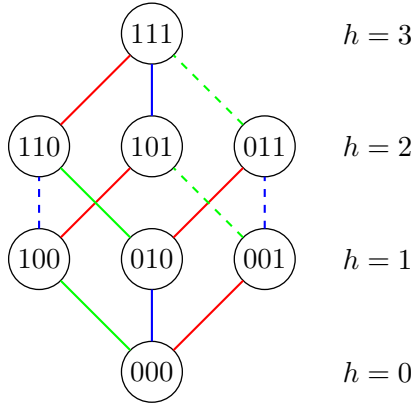

Many of our proofs involve algebraic manipulation. To make our treatment more streamlined, we now set up algebraic interpretations of our definitions.

\begin{itemize}
\item The condition of $A$ being a chromotopology is equivalent to having a map $q_i \colon V(A) \to V(A)$ for every color $i$ that sends each vertex $v$ to the unique vertex connected to $v$ by the edge with color $i$, such that the different $q_i$ commute (equivalently, the $q_i$ generate a free $\ZZ_2^n$ action on $V(A)$. The well-definedness of the $q_i$ corresponds to the edge-coloring condition and the commutation requirement corresponds to the $4$-cycle condition. Note that $q_i$ is an involution, as applying $q_i$ twice simply traverses the same edge twice. Furthermore, $q_i$ sends any boson to a fermion, and vice-versa.
\item The condition of a chromotopology $A$ being dashed (with dashing function $d$) is equivalent to having the maps $\Q_i$ anticommute, where we define $\Q_i\colon \RR[V(A)] \to \RR[V(A)]$ for every color $i$ by $\Q_i(v) = d(v, q_i(v)) q_i(v)$.
\end{itemize}

Finally, we define a few forgetful functions in the intuitive way: for any (possibly ranked and/or dashed) chromotopology $A$, we will use ``the chromotopology of $A$'' to mean the vertex- and edge-colored graph of $A$, forgetting all other information. We similarly say ``the topology of $A$'' to mean the graph of $A$ with no colors of any sort.

\subsection{Multigraphs}

In Section~\ref{sec:graph quotients}, multigraphs come up naturally. Thus, we generalize some of our definitions to multigraphs.

Let a \emph{pretopology} be an $n$-regular finite connected multigraph (that is, we allow loops and multiple-edges). Let a \emph{prechromotopology} be a generalization of \emph{chromotopology} where the graph is allowed to be a pretopology. The condition is still that the $q_i$ must commute. However, the combinatorial version of the rule (that the union of edges of different colors $i$ and $j$ form a disjoint union of $4$-cycles) must be extended to allow degenerate $4$-cycles that use double-edges or loops. Define the \emph{dashed} and \emph{ranked} properties on prechromotopologies and \emph{preadinkras} analogously, again extending our condition for $2$-colored $4$-cycles to allow double-edges and loops.

The double-edges actually do not introduce any new dashed prechromotopologies (and thus preadinkras), because the existence of a double-edge immediately gives a degenerate $4$-cycle, and the sum of dashes over a degenerate $4$-cycle must be even. Loops, however, create new preadinkras. Figure~\ref{fig:preadinkra} gives an example.

\begin{figure}[htb]
\begin{center}

\begin{tikzpicture}[scale=0.10]
%\SetVertexNormal
\SetUpEdge[labelstyle={draw}]
\Vertex[x=0,y=0]{0}
\Vertex[x=0,y=10]{1}
\Loop[dist=10cm, color=red, dir=SO, style=solid](0)
\Loop[dist=10cm, color=red, dir=NO, style=dashed](1)
\Edge[color=green, style={bend left}](0)(1)
\Edge[color=green, style={bend right}](0)(1)
\end{tikzpicture}

\caption{A preadinkra with two loops. \label{fig:preadinkra}}
\end{center}
\end{figure}
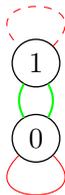

\section{Physical Motivation}
\label{sec:physical motivation}

The reader is already equipped to understand the rest of the paper (with the exception of Section~\ref{sec:together}) with no knowledge from this section. However, we hope our brief outline will serve as enrichment that may provide some additional intuition, as well as provide a review of the original problems of interest (where much remains to be done). While knowledge of physics will help in understanding this section, it is by no means necessary. We have neither the space nor the qualification to give a comprehensive review, so we encourage interested readers to explore the original physics literature.

The physics motivation for adinkras is the following: ``we want to understand off-shell representations of the $\N$-extended Poincar\'{e} superalgebra in the $1$-dimensional worldline.'' There is no need to understand what all of these terms mean\footnote{The author certainly does not.} to appreciate the rest of the discussion; we now sketch the thinking process that leads to adinkras.

Put simply, we are looking at the representations of the algebra $\po{1}{N}$ generated by $\N+1$ generators $Q_1, Q_2, \ldots, Q_\N$ (the \emph{supersymmetry generators}) and $H = i\partial_t$ (the \emph{Hamiltonian}), such that
\begin{align*}
\{Q_I, Q_J\} & = 2 \delta_{IJ}H, \\
[Q_I, H] & = 0.
\end{align*} 
Here, $\delta$ is the Kronecker delta, $\{A, B\} = AB + BA$ is the anticommutator, and $[A,B] = AB-BA$ is the commutator. We can also say that $\po{1}{N}$ is a \emph{superalgebra} where the $Q_i$'s are \emph{odd} generators and $H$ is an \emph{even} generator. Since $H$ is basically a time derivative, it changes the \emph{engineering dimension} (physics units) of a function $f$ by a single power of time when acting on $f$.

Consider $\RR$-valued functions $\{\phi_1, \ldots, \phi_m\}$ (the \emph{bosonic fields} or \emph{bosons}) and $\{\psi_1, \ldots, \psi_m\}$ (the \emph{fermionic fields} or \emph{fermions}), collectively called the \emph{component fields}\footnote{The fact that the two cardinalities match come from the physics assumption that the representations are \emph{off-shell}; i.e. the component fields do not obey other differential equations.}. We want to understand representations of $\po{1}{N}$ acting on the following infinite basis: 
$$\{H^k\phi_I, H^k\psi_J \mid k \in \NN; I, J \leq m\}.$$
There's a subtlety here, as these infinite-dimensional representations are frequently called ``finite-dimensional'' by physicists, who would just call the $\{\phi_I\}$ and the $\{\psi_I\}$ as the ``basis,'' emphasizing the finiteness of $m$. A careful treatment of this is given in \cite{d2l:supermodules}.) A long-open problem is:

\begin{open}
\label{que:all classification}
What are all such ``finite-dimensional'' representations of $\po{1}{N}$? What if we extend to higher dimensions (we will explain what this means to Section~\ref{sec:extensions})?
\end{open}

In particular, we want to understand representations of $\po{1}{N}$ satisfying some physics restrictions (most importantly, having the supersymmetry generators send bosons to only fermions, and vice-versa; this kind of ``swapping symmetry'' is what supersymmetry tries to study). Understanding all such representations seems intractible, so we restrict our attention to representations where the supersymmetry generators act as permutations (up to a scalar) and also possibly the Hamiltonian $H = i\partial_t$ on the basis fields: we require that for any boson $\phi$ and any $Q_I$,
$$ Q_I\phi = \pm (-iH)^s \psi = \pm (\partial_t)^s \psi,$$
where $s \in \{0, 1\}$, the sign, and the fermion $\psi$ depends on $\phi$ and $I$. We enforce a similar requirement
$$ Q_I\psi = \pm i(-iH)^s \phi = \pm i (\partial_t)^s \phi $$
for fermions.  We call the representations corresponding to these types of actions \emph{adinkraic representations}. For each of these representations, we associate an \emph{adinkra}. We now form a correspondence between our definition of adinkras in Section~\ref{sec:definitions} and our definition for adinkraic representations.

\begin{center}
\begin{tabular}{c|c}
adinkras & representations \\
\hline
vertex bipartition & bosonic/fermionic bipartition \\
colored edges and $q_I$ & action of $Q_I$ without the sign or powers of $(-iH)$ \\
dashing / sign in $\Q_I$ & sign in $Q_I$ \\
change of rank by $q_I$ and $\Q_I$ & powers of $(-iH)$ in $Q_I$ \\
rank function & partition of fields by engineering dimension
\end{tabular}
\end{center}

To summarize, 

\begin{framed}
\noindent An adinkra encodes a representation of $\po{1}{N}$. An adinkraic representation is a representation of $\po{1}{N}$ that can be encoded into an adinkra.
\end{framed}

So instead of attacking Question~\ref{que:all classification} head-on, we focus on the following problem instead:
\begin{open}
\label{que:adinkraic classification}
What are all the adinkraic representations of $\po{1}{N}$?
\end{open}

The set of adinkraic representations is already rich enough to contain representations of interest. When the poset structure of our adinkra $A$ is a boolean lattice, we get what \cite{d2l:topology} calls the \emph{exterior supermultiplet}, which coincides with the classical notion of \emph{superfield} introduced in \cite{salam:super-gauge}. When the poset of $A$ is a height-$2$ poset (in which case we say that $A$ is a \emph{valise}), we get \cite{d2l:topology}'s \emph{Clifford supermultiplet}. By direct sums, tensors, and other operations familiar to the Lie algebras setting, it is possible to construct many more representations (see \cite{d2l:topology} and \cite{d2l:clifford}), a technique that has been extended to higher dimensions in \cite{hubsch:weaving}.

\section{Topologies and Chromotopologies}
\label{sec:topologies}

In this section, we study prechromotopologies, chromotopologies, and adinkraizable chromotopologies. Compared to the relevant sections of \cite{d2l:topology} and \cite{d2l:clifford}, our approach is more general, though we owe many ideas to the original work. There is a pleasant connection to codes and Clifford algebras. 

We now give a quick review of codes (there are many references, including \cite{huffman}). See Appendix~\ref{app:clifford} for a review of Clifford algebras and some related results we will need. An $n$-\emph{bitstring} is a vector in $\ZZ_2^{n}$, which we usually write as $b_1b_2\cdots b_n$, $b_i \in \ZZ_2$. We distinguish two $n$-bitstrings ${\overrightarrow{1_n}} = 11\ldots 1$ and ${\overrightarrow{0_n}} = 00\ldots 0$, and when $n$ is clear from context we suppress the subscript $n$. The number of $1$'s in a bitstring $v$ is called the \emph{weight} of the string, which we denote by $\wt(v)$. We use $\overline{v}$ to denote the bitwise complement of $v$, which reverses $0$'s and $1$'s. For example, $\overline{00101} = 11010$. An $(n,k)$-\emph{linear binary code} (for this paper, we will not talk about any other kind of codes, so we will just say \emph{code} for short) is a $k$-dimensional $\ZZ_2$-subspace of bitstrings. A code is \emph{even} if all its bitstrings have weight divisible by $2$ and \emph{doubly even} if all its bitstrings have weight divisible by $4$.

%After the classification, we introduce two new ideas: one, we interpret \cite{d2l:topology}'s notion of \emph{graph quotient} as quotients on cell complexes; two, we introduce the notion of \emph{decomposition}. We will leverage both of these very simple ideas to prove original results in Sections~\ref{sec:dashing} and \ref{sec:ranking}.

%in Proposition~\ref{prop:labeled switching classes}, Proposition~\ref{thm:cube dashings}, Theorem~\ref{thm:general dashings}, and an algorithm in Section~\ref{sec:counting rankings}.

%The \emph{Hamming distance} between bitstrings $a$ and $b$, which we will denote by $|a - b|$, is the number of different bits between them; so $|1101 - 0110| = 3$. 

We now introduce the key running example throughout our paper. Define the $n$-dimensional \emph{hypercube} to be the graph with $2^n$ vertices labeled by the $n$-bitstrings, with an edge between two vertices if they differ by exactly one bit. This graph is bipartite and $n$-regular. Thus, it has the structure of a topology, which we can call the \emph{$n$-cubical topology}, or $I^n$. Now, if two vertices differ at some bit $i$, $1 \leq i \leq n$, color the edge between them with the color $i$. The $2$-colored $4$-cycle condition holds, so we get a chromotopology $I_c^n$, the \emph{$n$-cubical chromotopology}. Figure~\ref{fig:3cube norank} shows $I_c^3$; our earlier example adinkra in Figure~\ref{fig:3cube} also had this chromotopology. 

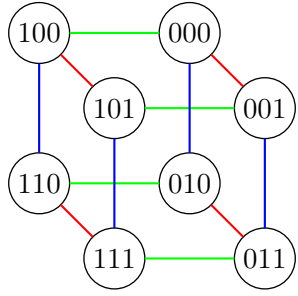
\begin{figure}[htb]
\begin{center}

\begin{tikzpicture}[scale=0.10]
%\SetVertexNormal
\SetUpEdge[labelstyle={draw}]
\Vertex[x=0,y=0]{111}
\Vertex[x=20,y=0]{011}
\Vertex[x=0,y=20]{101}
\Vertex[x=20,y=20]{001}
\Vertex[x=-10,y=10]{110}
\Vertex[x=10,y=10]{010}
\Vertex[x=-10,y=30]{100}
\Vertex[x=10,y=30]{000}
\Edge[color=red](100)(101)
\Edge[color=red](000)(001)
\Edge[color=red](010)(011)
\Edge[color=red](110)(111)
\Edge[color=green](000)(100)
\Edge[color=green](001)(101)
\Edge[color=green](010)(110)
\Edge[color=green](011)(111)
\Edge[color=blue](000)(010)
\Edge[color=blue](001)(011)
\Edge[color=blue](100)(110)
\Edge[color=blue](101)(111)
\end{tikzpicture}
\caption{The hypercube chromotopology $I_c^3$. \label{fig:3cube norank}}
\end{center}
\end{figure}

\subsection{The Valise}
\label{sec:valise}

Note that any bipartite prechromotopology (including all chromotopologies) $A$ can be ranked as follows: take one choice of bipartition of $V(A)$ into bosons and fermions. Assign the rank function $h$ to take values $0$ on all bosons and $1$ on all fermions, which creates a height-$2$ poset. We call the corresponding ranked prechromotopology a \emph{valise}. Because we could have switched the roles of bosons and fermions, each bipartite prechromotopology gives rise to exactly two valises.

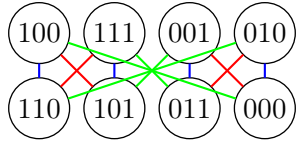
\begin{figure}[htb]
\begin{center}

\begin{tikzpicture}[scale=0.10]
%\SetVertexNormal 
\SetUpEdge[labelstyle={draw}]
\Vertex[x=0,y=20]{111}
\Vertex[x=0,y=10]{101}
\Vertex[x=20,y=20]{010}
\Vertex[x=20,y=10]{000}
\Vertex[x=-10,y=10]{110}
\Vertex[x=10,y=10]{011}
\Vertex[x=-10,y=20]{100}
\Vertex[x=10,y=20]{001}
\Edge[color=red](100)(101)
\Edge[color=red](000)(001)
\Edge[color=red](010)(011)
\Edge[color=red](110)(111)
\Edge[color=green](000)(100)
\Edge[color=green](001)(101)
\Edge[color=green](010)(110)
\Edge[color=green](011)(111)
\Edge[color=blue](000)(010)
\Edge[color=blue](001)(011)
\Edge[color=blue](100)(110)
\Edge[color=blue](101)(111)
\end{tikzpicture}

\caption{A valise with topology $I^3$.\label{fig:3cube valise}}
\end{center}
\end{figure}

\subsection{Graph Quotients and Codes}
\label{sec:graph quotients}
In this section, we recover the main result (Theorem~\ref{thm:adinkraizable as quotient}) classifying adinkraizable chromotopologies from the existing literature. However, we take a more general approach with multigraphs\footnote{The existing literature only considers quotients which are simple graphs, though double edges occur naturally in the quotient. An alternate way to avoid multigraphs would be to define the quotient to identify differently-colored multiple-edges as a single edge, but then the number of colors would change under a quotient and we lose information. We take the following compromise: for classification, we use our approach since we feel it is the most natural and inclusive setup, but afterwards we will respect the original literature and focus on adinkraizable chromotopologies (which have simple graphs).} and also classify prechromotopologies and chromotopologies.

Consider the $n$-cubical chromotopology $I^n_c$. For any linear code $L \subset \ZZ_2^n$, the quotient $\ZZ_2^n / L$ is a $\ZZ_2$-subspace. Using this, we define the map $p_L$, which sends $I^n_c$ to the following prechromotopology, which we call the \emph{graph quotient} (or \emph{quotient} for short) $I^n_c/L$: 
\begin{itemize}
\item let the vertices of $I^n_c/L$ be labeled by the equivalence classes of $\ZZ_2^n/L$ and define $p_L(v)$ to be the image of $v$ under the quotient $\ZZ_2^n/L$. When $L$ is an $(n,k)$-code, the preimage over every vertex in $I^n_c/L$ contains $2^k$ vertices, so $I^n_c/L$ has $2^{n-k}$ vertices. 
\item let there be an edge $p_L(v,w)$ in $I^n_c/L$ with color $i$ between $p_L(v)$ and $p_L(w)$ in $I^n/L$ if there is at least one edge with color $i$ of the form $(v', w')$ in $\ZZ_2^n$, with $v' \in p_L^{-1}(v)$ and $w' \in p_L^{-1}(w)$.
\end{itemize}
Every vertex in $I^n_c/L$ still has degree $n$ (counting possible multiplicity) and the commutivity condition on the $q_i$'s is unchanged under a quotient, so $I^n_c/L$ is indeed a prechromotopology. Let the pretopology $I^n/L$ be the underlying multigraph of $I^n_c/L$. We now prove some properties of the quotient.

% \begin{lemma}
% \label{lem:dashing code}
% A code is a dashing code 
% \end{lemma}
% \begin{proof}
% \com{TODO}
% \end{proof}

\begin{prop}
\label{prop:codes}
The following hold for $A = I^n/L$, where $L$ is a code.
\begin{enumerate}
\item $A$ has a loop if and only if $L$ contains a bitstring of weight $1$; $A$ has a double edge if and only if $L$ contains a bitstring of weight $2$.
\item $A$ can be ranked if and only if $A$ is bipartite, which is true if and only if $L$ is an even code.%\footnote{In \cite{d2l:topology}, the restriction to even codes was enforced by a physical constraint (that the preimage of a vertex in the quotient must contain only bosons or only fermions). Here we stress that the condition already arises as a mathematical constraint.}. 
\end{enumerate}
\end{prop}

\begin{proof}
\begin{enumerate}
\item Suppose $A$ has a loop. This means some edge $(v,w)$ in $I^n_c$ has both endpoints $v$ and $w$ mapped to the same vertex in the quotient. Equivalently, $(v-w) \in L$. However, $v$ and $w$ differ by a bitstring of weight $1$. Suppose $A$ has a double edge $(v,w)$ with colors $i$ and $j$. Since $q_1(q_2(v)) = v$ in $A = I^n_c/L$, for some $v' \in p_L^{-1}(v)$, we must have in $I^n_c$ that $q_1(q_2(v')) - v'$ is in $L$. But this is a weight $2$ bitstring with support in $i$ and $j$. The logic is reversible in both of these situations.

\item Suppose $A$ were not bipartite, then $A$ has some odd cycle. One of the preimages of this cycle in $I_c^n$ is a path of odd length from some $v$ to some $w$ that both map to the same vertex under the quotient (i.e. $v - w \in L$). Since each edge changes the weight of the vertex by $1 \pmod{2}$, $v - w$ must have an odd weight. Since $v - w \in L$, $L$ cannot be an even code. In the other direction, if $L$ were an even code such odd cycles cannot occur. 

Recall that any bipartite prechromotopology can be ranked by making a valise (see Section~\ref{sec:valise}). If $A$ can be ranked via a rank function $h$, the sets $\{v \in V(A) \mid h(v) \cong 0 \pmod{2}\}$ and $\{v \in V(A) \mid h(v) \cong 1 \pmod{2}\}$ must be a bipartition of $A$ because all the edges in $A$ change parity of $h$. \qedhere
\end{enumerate}
\end{proof}

The most difficult condition to classify is dashing. For this proof, we need the material from Appendix~\ref{app:clifford}. Let a \emph{dashing code} be a code where the following two conditions hold: first, all bitstrings in the code must have weight $0$ or $1 \pmod{4}$; second, for any two bitstrings $w_1$ and $w_2$, we have
$$(w_1 \cdot w_2) + \wt(w_1) \wt(w_2) = 0 \pmod{2},$$
where the first term is the dot product in $\ZZ_2^n$. The following result extends the combination of ideas used in {\cite[Construction 3.1]{d2l:clifford}} and {\cite[Theorem 4.2]{d2l:topology}}. 

\begin{prop}
\label{prop:dashing codes}
The prechromotopology $A = I^n_c/L$ can be dashed if and only if $L$ is a dashing code.
\end{prop}
\begin{proof}
Given a bitstring $v = v_1 v_2 \cdots v_n$, define $\Q_v$ to be the map $\Q_n^{v_n} \cdots \Q_1^{v_1}$. 

Suppose we have an odd dashing. Let $v$ and $w$ be codewords in $L$. Both $\Q_v$ and $\Q_w$ take any vertex to itself in $\RR[V(A)]$ with possibly a negative sign, since the $\Q_i$ are basically the $q_i$ with possibly a sign, and following a sequence of $q_i$ corresponding to a codeword is a closed loop in $I_c^n/L$. This means  $\Q_v$ and $\Q_w$ must commute; furthermore, $\Q_v^2$ must be the identity map for any $v \in L$. By Lemma~\ref{lem:clifford commutation}, this is exactly the condition required for $L$ being a dashing code.

Now, suppose $L$ were a dashing code. Then by Proposition~\ref{prop:clifford weird group}, we can find a sign function $s$ such that $\{s(v)\clif(v) \mid v \in L\}$ form a subgroup $\smon_L \subset \smon$, the signed monomials of $\cl(n)$. This gives a well-defined action of $\smon$ on $\smon / \smon_L$ via left multiplication while possibly introducing signs. The cosets of $\smon_L$ under $\smon$ naturally correspond to $V(A)$, so we can define $\Q_i(v)$ to introduce the same sign as $\gamma_i$ on $\clif(v) \in \smon / \smon_L$. Since we have a Clifford algebra action, we get the desired anticommutation relations between $\Q_i$ and thus an odd dashing.
\end{proof}

%The graph $I^n_c/L$ is simple if and only if no two edge colors appear between two vertices in $I^n_c/L$; namely for any edge $(v', w')$ where $p_L(v') = v$ and $p_L(w') = w$, $v'$ and $w'$ must differ in the same bit. Suppose this condition were not met. 

Quotients of $I^n_c$ are prechromotopologies. Surprisingly, the converse is also true, which gives us our main classification:

\begin{theorem}
\label{thm:multigraphs are quotients}
Prechromotopologies are exactly quotients $I^n_c/L$ for some code $L$.
\end{theorem}
\begin{proof}
Take a prechromotopology $A$. Consider the abelian group $G$ acting on $V(A)$ generated by the $q_i$. The elements of $G$ can be written as $g = q_1^{s_1} q_2^{s_2} \cdots q_n^{s_n}$, where $s_i \in \ZZ_2$ for all $i$. Consider the isomorphism $\phi\colon G \to L$ which sends such a $g$ to the $n$-bitstrings $s_1 s_2 \cdots s_n \in \ZZ_2^n$. Take any vertex $v_0 \in V(A)$ and consider its stabilizer group $H$ under $G$. $\phi(H)$ is a subspace of $\ZZ_2^n$ and thus must be some code $L$. Any vertex $v$ is equal to $g(v_0)$ for some $g \in G$, so we may label $v$ with the coset $\phi(g) + L$. It is easy to check the resulting prechromotopology is exactly the one produced by the quotient $I^n_c/L$.
\end{proof}

Combining Proposition~\ref{prop:codes} and Theorem~\ref{thm:multigraphs are quotients} immediately gives the classification of all chromotopologies and adinkraizable chromotopologies:

\begin{theorem}
\label{thm:chromotopology as quotient}
Chromotopologies are exactly quotients $I^n_c/L$, where $L$ is an even code with no bitstring of weight $2$.
\end{theorem}

\begin{theorem}[DFGHILM, {\cite[Theorem 4.1]{d2l:topology}} and {\cite[Section 3.1]{d2l:clifford}}]
\label{thm:adinkraizable as quotient}
Adinkraizable chromotopologies are exactly quotients $I^n_c/L$, where $L$ is a doubly even code.
\end{theorem}

Thanks to Theorem~\ref{thm:multigraphs are quotients}, we can assume the following:
\begin{framed}
\noindent From now on, any prechromotopology (including chromotopologies) $A$ we discuss comes from some $(n,k)$-code $L(A) = L$. If $L$ is an $(n,k)$-code, we say that the corresponding $A$ is an $(n,k)$-prechromotopology (or chromotopology).
\end{framed}

An $(n,0)$-prechromotopology is exactly the $n$-cubical chromotopology, corresponding to the trivial code $\{\overrightarrow{0}\}$. The first non-cubical chromotopology, shown in Figure~\ref{fig:4cube folding}, is the result of quotienting the $4$-cubical topology by the code $L = \{0000, 1111\}$, the smallest non-trivial doubly-even code. It has the topology of the bipartite graph $K_{4,4}$.

\begin{figure}[htb]
\begin{center}

\begin{tabular}{cc}
\begin{tikzpicture}[scale=0.10]
%\SetVertexNormal
\SetUpEdge[labelstyle={draw}]
\Vertex[x=0,y=0]{A}
\Vertex[x=-5,y=10]{C}
\Vertex[x=-15,y=10]{B}
\Vertex[x=5,y=10]{D}
\Vertex[x=15,y=10]{E}
\Vertex[x=-25,y=20]{H'}
\Vertex[x=-15,y=20]{G'}
\Vertex[x=-5,y=20]{F'}
\Vertex[x=25,y=20]{H}
\Vertex[x=15,y=20]{G}
\Vertex[x=5,y=20]{F}
\Vertex[x=-5,y=30]{D'}
\Vertex[x=-15,y=30]{E'}
\Vertex[x=5,y=30]{C'}
\Vertex[x=15,y=30]{B'}
\Vertex[x=0,y=40]{A'}
\Edges(A,B,H',C,A,D,H,E,A)
\Edges(B,G',D,F,C,G,E,F',B)
\Edges(A',B',H,C',A',D',H',E',A')
\Edges(B',G,D',F',C',G',E',F,B')
\end{tikzpicture}
&
\begin{tikzpicture}[scale=0.10]
%\SetVertexNormal
\SetUpEdge[labelstyle={draw}]
\Vertex[x=0,y=0]{A}
\Vertex[x=-5,y=10]{C}
\Vertex[x=-15,y=10]{B}
\Vertex[x=5,y=10]{D}
\Vertex[x=15,y=10]{E}
%\Vertex[x=-25,y=20]{H'}
%\Vertex[x=-15,y=20]{G'}
%\Vertex[x=-5,y=20]{F'}
\Vertex[x=25,y=20]{H}
\Vertex[x=15,y=20]{G}
\Vertex[x=5,y=20]{F}
%\Vertex[x=-5,y=30]{D'}
%\Vertex[x=-15,y=30]{E'}
%\Vertex[x=5,y=30]{C'}
%\Vertex[x=15,y=30]{B'}
%\Vertex[x=0,y=40]{A'}
\Edges(A,B,G)
\Edges(B,H,C)
\Edges(C,A,D,H,E,A)
\Edges(D,F,C,G,E,F,B)
%\Edges(A',B',H,C',A',D',H',E',A')
%\Edges(B',G,D',F',C',G',E',F,B')

\end{tikzpicture}
\end{tabular}
\caption{The topologies $I^4$ and $I^4/\{0000,1111\}$. Labels with the same letter are sent to the same vertex. \label{fig:4cube folding}}
\end{center}
\end{figure}
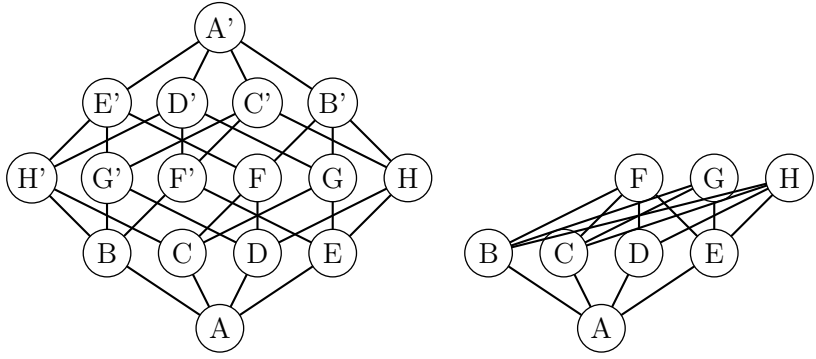

While we have ``solved'' the problem of classifying adinkraizable chromotopologies by reducing it to that of classifying doubly-even linear codes, the theory of these codes is very rich and nontrivial. Computationally, \cite{miller} contains the current status of the classification through an exhaustive search. Interestingly, when examining irreducible adinkratic representations, self-orthogonal codes come up. Self-orthogonal codes form a well-studied subset of codes. We talk more about this in Section~\ref{sec:together}.

\subsection{A Homological View}

From the theorems in Section~\ref{sec:graph quotients}, we see that the $n$-cubical chromotopologies $I^n_c$ look like universal covers in the sense that everything else come from their quotients. We make this intuition rigorous in this section with the language of homological algebra. Any standard introduction, such as \cite{Hatcher}, is more than sufficient for our purposes.

We work over $\ZZ_2$. Construct the following $2$-dimensional complex $X(A)$ from a chromotopology $A$. Let $C_0$ be formal sums of elements of $V(A)$ and $C_1$ be formal sums of elements of $E(A)$. For each $2$-colored $4$-cycle $C$ of $A$, create a $2$-cell with $C$ as its boundary as a generator in $C_2$, the boundary maps $\{d_i\colon C_i \to C_{i-1}\}$ are the natural choices (we do not worry about orientations since we are using $\ZZ_2$), giving homology groups $H_i = H_i(X(A))$. The most important observation about our complex $X(A)$ is the following, which we return to in Section~\ref{sec:homology}.

\begin{prop}
\label{prop:covering space}
Let $A$ be an $(n,k)$-adinkraizable chromotopology corresponding to the code $L$. Then $X(A) = X(I_c^n)/L$ as a quotient complex, where $L$ acts freely on $X(I_c^n)$. We have that $X(I_c^n)$ is a simply-connected covering space of $X(A)$, with $L$ the group of deck transformations.
\end{prop}
\begin{proof}
The fact that $X(A)$ is a quotient complex is already evident from the construction of the graph quotient, since we have restricted to simple graphs (recall that adinkraizable chromotopologies have simple graphs). The more interesting statement is that $X(I^n)$ is simply connected. A cute way to see this is to note that $X(I^n)$ is the $2$-skeleton of the $n$-dimensional (solid) hypercube $Y$. Thus, $X(I^n)$ and $Y$ must have matching $H_1$ and $\pi_1$. But $Y$ is obviously simply-connected.
\end{proof}

\subsection{Adinkra Decomposition}
\label{sec:decomposition}

Finally, we introduce a notion designed to reduce the complexity of chromotopologies. Say that a color $i$ \emph{decomposes} a chromotopology $A$ if removing all edges of color $i$ splits $A$ into $2$ separate connected components. Our definition was inspired by observations in \cite{d2l:clifford}, where certain adinkras were called \emph{$1$-decomposable}.

%First, assume we have a set of colors labeled by $\NN$, and impose on them the natural ordering $c_1 < c_2 < \cdots$. Take an $(n,k)$-adinkra $A$ where the $n$ bits correspond to the colors $(c_1, c_2, \ldots, c_n)$ in order, where $c_1 < c_2 < \cdots < c_n$ and all $c_i \in \NN$ (this is one of the few times we will make colors explicit). 

\begin{lemma}
\label{lem:cut preserves bit}
The color $i$ decomposes the chromotopology $A$ if and only if for all $d \in L(A)$, the $i$-th bit of $d$ is $0$. %Equivalently, all elements in the equivalence class of a vertex of $A$ have the same $i$-th bit.
\end{lemma}

\begin{proof}
The condition that the $i$-th bits of all codes in $L(A)$ are $0$ is equivalent to the condition that the vertices in $I_c^n/L$ correspond to equivalence classes $v = \{v_i\}$ where all of the $v_i$ have the same $i$-th bit. Thus, when this condition is met we can split the vertices in $I_c^n/L$ into two classes, $A_0$ and $A_1$, where vertices in $A_0$ have $0$ as the $i$-th bit and vertices in $A_1$ have $1$ as the $i$-th bit. Now, if vertices $v \in A_0$ and $w \in A_1$ have an edge corresponding to bit $j \neq i$, they must correspond to two bitstrings $v'$ and $w'$ where $q_j(v') = w'$. But this is impossible, since we know all elements in $p_L^{-1}(v)$ have $0$ in the $i$-th bit and the opposite is true for $p_L^{-1}(w)$. Thus, such two vertices $v$ and $w$ can only have edges corresponding to bit $i$, and this is equivalent to saying that $i$ decomposes $A$.

On the other hand, if $i$ decomposes $A$, let the two connected components have vertex sets $A_0$ and $A_1$, take some edge $(v, w)$ in $A$ with $v \in A_0$ and $w \in A_1$. The edge must have color $i$. Now, take a preimage $v' \in p_L^{-1}(v)$. If any bitstring $d \in L(A)$ with weight $k$ has $1$ as the $i$-th bit, then we can get from $v'$ to some $w' \in A$ via $k-1$ steps corresponding to the colors in the support of $d$ that are not $i$. Under $p_L$, this walk sends $w'$ to $w$, so we have a connected path between $v$ and $w$ in $A$ without using color $i$, a contradiction since $i$ decomposes $A$.
\end{proof}

\begin{cor}
\label{cor:everything cuts hypercube}
Every color in $[n]$ decomposes $I_c^n$.
\end{cor}

In the situation where Lemma~\ref{lem:cut preserves bit} holds, we say that $i$ decomposes $A$ into $\piece{A}{i}{0}$ and $\piece{A}{i}{1}$, or $\cut{A}{\piece{A}{i}{0}}{\piece{A}{i}{1}}{i}$, if removing all edges with color $i$ creates two disjoint chromotopologies $\piece{A}{i}{0}$ and $\piece{A}{i}{1}$, which are labeled and colored in a natural fashion, equipped with an inclusion $\inc$ on their vertices that map into $V(A)$. Formally: % and a restriction $\res$ in the other direction. Formally:
\begin{itemize}
\item $V(A)$ can be parititioned into two sets $V(A|0)$ and $V(A|1)$, where vertices in $V(A|0)$ have $0$ in the $i$-th bit (by Lemma~\ref{lem:cut preserves bit}, this is a well-defined notion) and vertices in $V(A|1)$ have $1$ in the $i$-th bit. Furthermore, all edges between $V(A|0)$ and $V(A|1)$ are of color $i$.
\item define $\piece{A}{i}{0}$ to be isomorphic to the edge-colored graph induced by vertices in $V(A|0)$, where any bitstring $v = (b_1 b_2 \cdots b_n)$ in the vertex label class of $v' \in V(A|0)$ is sent to the $(n-1)$-bitstring $(b_1 b_2 \cdots \widehat{b_i} \cdots b_n)$, where we remove the bit $b_i$. Color the edges analogously with colors in $\{1, 2 \cdots, \widehat{i}, \cdots, n\}$. Define $\piece{A}{i}{1}$ in the same way with $V(A|1)$ instead of $V(A|0)$.
\item define the maps $\inc(b_1 b_2 \ldots b_{n-1}, j \to i) = b_1 \ldots b_{i-1} j b_i \ldots b_{n-1}$, which inserts $j$ into the $i$-th place of an $(n-1)$-bitstring to create an $n$-bitstring. If $\cut{A}{\piece{A}{i}{0}}{\piece{A}{i}{1}}{i}$, let $\inc(v)$ send a vertex $v \in \piece{A}{i}{j}$ to $\inc(v, j \to i)$ for $j \in \{0, 1\}$. Lemma~\ref{lem:cut preserves bit} gives that the union of the image of $V(\piece{A}{i}{0})$ and $V(\piece{A}{i}{1})$ under $\inc$ is exactly $V(A)$. %Define $\res(v, j)$, $v \in V(A)$ and $j \in \{0,1\}$, to be the inverse map, which takes $v$ into $V(\piece{A}{c_i}{j})$.
\end{itemize}

% and that $\piece{A}{c_i}{0}$ and $\piece{A}{c_i}{1}$ have height functions $h_0$ and $h_1$ respectively. 
If $A$ has the additional structure of a ranked chromotopology, we can say more. Suppose $\cut{A}{\piece{A}{i}{0}}{\piece{A}{i}{1}}{i}$. Now, let $z_0 = \inc({\overrightarrow{0_{n-1}}}, 0 \to i)$ and $z_1 = \inc({\overrightarrow{0_{n-1}}}, 1 \to i)$ be elements in $V(A)$. Since they are adjacent, their rank functions must differ by exactly $1$; that is, $|h(z_0) - h(z_1)| = 1$. We denote $A = \piece{A}{i}{0} \nearrow_{i} \piece{A}{i}{1}$ in the case $h(z_1) = h(z_0) + 1$ and $A = \piece{A}{i}{0} \searrow_{i} \piece{A}{i}{1}$ otherwise. See Figure~\ref{fig:3cubecut} for an example.
\begin{comment}
\begin{itemize}
\item if $A = \piece{A}{i}{0} \nearrow_{i} \piece{A}{i}{1}$, we can construct $A$ by matching up the ranks of the $z_i$ in $\piece{A}{i}{0}$ and $\piece{A}{i}{1}$, shifting those in $\piece{A}{i}{1}$ up by $1$; 
\item if $A = \piece{A}{i}{0} \searrow_{i} \piece{A}{i}{1}$. we switch the roles of $\piece{A}{i}{0}$ and $\piece{A}{i}{1}$ in the above.
\end{itemize}
\end{comment}
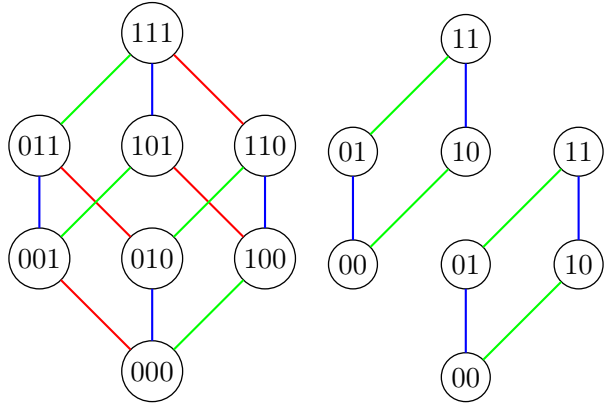
\begin{figure}[htb]
\begin{center}

\begin{tabular}{cc}
\begin{tikzpicture}[scale=0.15]
%\SetVertexNormal
\SetUpEdge[labelstyle={draw}]
\Vertex[x=0,y=0]{000}
\Vertex[x=0,y=10]{010}
\Vertex[x=0,y=20]{101}
\Vertex[x=0,y=30]{111}
\Vertex[x=-10,y=10]{001}
\Vertex[x=10,y=10]{100}
\Vertex[x=-10,y=20]{011}
\Vertex[x=10,y=20]{110}
\Edge[color=red](100)(101)
\Edge[color=red](000)(001)
\Edge[color=red](010)(011)
\Edge[color=red](110)(111)
\Edge[color=green](000)(100)
\Edge[color=green](001)(101)
\Edge[color=green](010)(110)
\Edge[color=green](011)(111)
\Edge[color=blue](000)(010)
\Edge[color=blue](001)(011)
\Edge[color=blue](100)(110)
\Edge[color=blue](101)(111)
\end{tikzpicture}
 &

\begin{tikzpicture}[scale=0.15] 
%\SetVertexNormal
\SetUpEdge[labelstyle={draw}]
\Vertex[x=0,y=0]{00}
\Vertex[x=0,y=10]{01}
\Vertex[x=0,y=20]{10{}} % Best hack ever!
\Vertex[x=0,y=30]{11{}}
\Vertex[x=-10,y=10]{00{}}
\Vertex[x=10,y=10]{10}
\Vertex[x=-10,y=20]{01{}}
\Vertex[x=10,y=20]{11}
\Edge[color=green](11)(01)
\Edge[color=green](11{})(01{})
\Edge[color=green](10)(00)
\Edge[color=green](10{})(00{})
\Edge[color=blue](11)(10)
\Edge[color=blue](11{})(10{})
\Edge[color=blue](01)(00)
\Edge[color=blue](01{})(00{})
\end{tikzpicture}

\end{tabular}

\parbox{5in}{\caption{The color $3$ decomposes a ranked chromotopology $A$ (with chromotopology $I_c^3$) as $\piece{A}{3}{0} \nearrow_3 \piece{A}{3}{1}$ (we have $\nearrow$ because $000 > 001$ in $A$), each with chromotopology $I_c^2$.}
\label{fig:3cubecut}}
\end{center}
\end{figure}

\begin{prop}
\label{prop:cutting properties} Let $\cut{A}{A_0}{A_1}{i}$, where $A$ is an $(n,k)$-chromotopology. Then $A_0$ and $A_1$ are $(n-1, k)$ chromotopologies, isomorphic as graphs.
\end{prop}
\begin{proof}
The image of $q_i$ on $V(A_0$ is exactly $V(A_1)$ and $q_i$ is an involution, so we have a bijection between the vertices. If $q_j(v_1) = v_2$ in $A_0$, the $4$-cycle condition on $(v_1, q_i(v_1), q_i(v_2), v_2)$ gives that $(q_i(v_1), q_i(v_2))$ is also an edge of color $j$ in $A_1$, so the bijection between the vertices extends to a bijection between $A_0$ and $A_1$ as edge-colored graphs, and thus chromotopologies. Each of these chromotopologies has $2^{n-1}$ vertices and is $(n-1)$-regular, so by Proposition~\ref{prop:codes} they must be $(n-1, k)$- chromotopologies.
\end{proof}

\section{Dashing}
\label{sec:dashing}

Given an adinkraizable chromotopology $A$, define $o(A)$ to be the set of odd dashings of $A$.

\begin{open}
What are the enumerative and algebraic properties of $o(A)$? 
\end{open}

We introduce the concept of \emph{even dashings} and relate them to odd dashings, showing that not only do the even dashings form a more convenient model for calculations, there is a bijection between the two types of dashings. We will then count the number of (odd or even) dashings of $I^n_c$ via a cute application of linear algebra and generalize to all chromotopologies with a homological algebra computation.

Finally, we remark that studying dashed chromotopologies is basically equivalent to studying Clifford algebras. We discuss this further in Section~\ref{sec:clifford representations}. 

\subsection{Odd and Even Dashings}

Given an adinkraizable chromotopology $A$, even $I_c^n$, it is not intuitive if an odd dashing exists. It is somewhat surprising that one always does, given Proposition~\ref{prop:codes}. Let an \emph{even dashing} be a way to dash $E(A)$ such that every $2$-colored $4$-cycle contains an even number of dashed edges, and let $e(A)$ be the set of even dashings. We have the following nice fact:
\begin{lemma}
\label{lem:parity irrelevant}
For any adinkraizable chromotopology $A$, we have $|o(A)| = |e(A)|$.
\end{lemma}
\begin{proof}
Let $l = |E(A)|$. We may consider a dashing of $A$ as a vector in $\ZZ_2^l$, where each coordinate corresponds to an edge and is assigned $1$ for a dashed edge and $0$ for a solid edge. There's an obvious way to add two dashings (i.e. addition in $\ZZ_2^l$) and there is a zero vector $d_0$ (all edges solid), so the family of all dashings (with no restrictions) form a vector space $V_{\text{free}}(n)$ of dimension $l$.

Now, the crucial observation is that $e(A)$ has the structure of a subspace of $V_{\text{free}}(n)$. To see this, we can directly check that adding two even dashings preserve the even parity of each $2$-colored $4$-cycle and that $d_0$ is an even dashing. Alternatively, we can note the restriction of a dashing $d$ having a particular cycle with an even number of dashes just means the inner product of $d$ and some vector with four 1's as support is zero, so such dashings are exactly the intersection of $V_{\text{free}}(n)$ and a set of hyperplanes, which is a subspace. 

Unlike the even dashings $e(A)$, the odd dashings $o(A)$ do not form a vector space; in particular, they do not include $d_0$. However, adding an even dashing to an odd dashing gives an odd dashing and the difference between any two odd dashings gives an even dashing. Thus, $o(A)$ is a coset in $V_{\text{free}}(n)$ of $e(A)$ and must then have the same cardinality as $e(A)$ given that at least one odd dashing exists. Since $A$ is adinkraizable by definition, we are done.
\end{proof}

The proof of Lemma~\ref{lem:parity irrelevant} shows that the odd dashings form a \emph{torsor} for the even dashings, which are easier to work with. %. Even dashings are in many ways easier. For example, ``trivial'' constructions such as having no dashed edges or having only dashed edges are already even dashings, while existence and construction for odd dashings are not obvious. Furthermore, the even dashings have a natural structure of a $\ZZ_2$-vector space.

\subsection{Decompositions and Dashing $I_c^n$}

We start by looking at the $n$-cubical chromotopology $I_c^n$. The main simplification here is that dashings behave extremely well under decompositions.

\begin{lemma}
\label{lem:even dashings and cuts}
If $A$ has $l$ edges colored $i$ and $\cut{I_c^n}{A_0}{A_1}{1}$, then each even (resp. odd) dashing of the induced graph of $A_0$ and each of the $2^l$ choices of dashing the $i$-colored edges extends to exactly one even (resp. odd) dashing of $A$.
\end{lemma}
\begin{proof}
Without loss of generality, we can take $i = 1$, so $A_0$ contains equivalence classes of bitstrings with first bit $0$ and $A_1$ contains those with first bit $1$.

After an even dashing of $A_0$ and an arbitrary dashing of the $i$-colored edges, note the remaining $2$-colored $4$-cycles are of exactly two types:
\begin{enumerate}
\item the $4$-cycles in $A_1$;
\item the $4$-cycles of the form $(u,v,w,x)$, where $(u,v)$ is in $A_0$, $(w,x)$ is in $A_1$, and $(v,w)$ and $(x,u)$ are colored $i$. 
\end{enumerate}

Note that in all the cycles $(u,v,w,x)$ of the second type, $(w,x)$ is the only one we have not selected. Thus, there is exactly one choice for each of those edges to satisfy the even parity condition. Since there is exactly one such cycle for every edge in $A_1$, this selects a dashing for all the remaining edges, and the only thing we have to check is that the $4$-cycles of the first type, the ones entirely in $A_1$, are evenly dashed. 

Now, a $4$-cycle of this type is of form $(1a_1, 1a_2, 1a_3, 1a_4)$, which is a face of a hypercube with vertices $(0a_1, 0a_2, 0a_3, 0a_4, 1a_1, 1a_2, 1a_3, 1a_4)$. There are $5$ other $4$-cycles in this hypercube which have all been evenly dashed (the $0a_i$ vertices form a cycle in $A_0$ and the other $4$ cycles are evenly dashed by our previous paragraph). Thus, we have that:
\begin{align*}
d(0a_1, 0a_2) + d(0a_2, 0a_3) + d(0a_3, 0a_4) + d(0a_4, 0a_1) & = 0; \\
d(0a_1, 0a_2) + d(0a_2, 1a_2) + d(1a_2, 1a_1) + d(1a_1, 0a_1) & = 0; \\
d(0a_2, 0a_3) + d(0a_3, 1a_3) + d(1a_3, 1a_2) + d(1a_2, 0a_2) & = 0; \\
d(0a_3, 0a_4) + d(0a_4, 1a_4) + d(1a_4, 1a_3) + d(1a_3, 0a_3) & = 0; \\
d(0a_4, 0a_1) + d(0a_1, 1a_1) + d(1a_1, 1a_4) + d(1a_4, 0a_4) & = 0.
\end{align*}
Adding these equations in $\ZZ_2$ gives:
\[
d(1a_1, 1a_2) + d(1a_2, 1a_3) + d(1a_3, 1a_4) + d(1a_4, 1a_1) = 0.
\]
Thus, we have constructed an even dashing. The analogous result for odd dashings follow by the same logic if we replace $0$'s on the right sides of the above equations by $1$'s.
\end{proof}

\begin{comment}

The rest of this section will serve as the proof of Theorem~\ref{thm:dash multiplier}. 

First, let's verify it for a couple of base cases:
\begin{enumerate}
\item For $n = 1$ we have a single edge. We can color this $2^{2^1-1} = 2$ ways.
\item For $n = 2$ we have a $4$-cycle, we can dash in $2^{2^2-1} = 8$ ways: $4$ ways for having a single dashed edge and $4$ ways for having $3$ dashed edges.
\end{enumerate}
\end{comment}

\begin{prop}
\label{thm:cube dashings}
The number of even (or odd) dashings of $I_c^n$ is  
\[
|e(I_c^n)| = |o(I_c^n)| = 2^{2^n-1}.
\]
\end{prop}
\begin{proof}
A convenient property of hypercubes is that every $4$-cycle is a $2$-colored $4$-cycle. Thus, we get to just say ``$4$-cycles'' instead of ``$2$-colored $4$-cycles'' in this proof.

We prove our result by induction. The base case is easy: for $n=1$ (a single edge), there are exactly $2$ even dashings, since there is no $4$-cycle. Suppose our result were true for every $k < n$. We will now show it is also true for $n$. Recall from Corollary~\ref{cor:everything cuts hypercube} that every color decomposes $I_c^n$. Let $\cut{I_c^n}{A^0}{A^1}{1}$, where both $A^i$ have the topology of $I_c^{n-1}$.

Since we have $2^{n-1}$ edges with color $1$, by Lemma~\ref{lem:even dashings and cuts} we get the recurrence
\[
|e(I_c^n)| = 2^{2^{n-1}}|e(I_c^{n-1})|.
\]
With the initial case $|e(I_c^1)| = 2$, we get $|e(I_c^n)| = 2^{2^{n-1} + 2^{n-2} + \cdots + 1} = 2^{2^n-1}$, as desired. The result for $|o(I_c^n)|$ is immediate by Lemma~\ref{lem:parity irrelevant}.
\end{proof}

Note that $|o(I_c^n)| = 2^{2^n}/2$. This suggests that, besides a single factor of $2$, each of the $2^n$ vertices gives exactly one ``degree of freedom'' for odd dashings. We will justify this hunch in the following discussion, in particular with Proposition~\ref{prop:switching freedom}.

\subsection{Vertex Switching}

In \cite{douglas}, Douglas, Gates, and Wang examined dashings from a point of view inspired by Seidel's \emph{two-graphs} (\cite{seidel:survey}). Define the \emph{vertex switch} at a vertex $v$ of a dashed chromotopology $A$ as the operation that produces the same $A$, except with all edges adjacent to $v$ flipped in parity (sending dashed edges to solid edges, and vice-versa). It is routine to verify that a vertex switch preserves odd dashings (in fact, parity in all $4$-cycles), so the odd dashings of $A$ can be split into orbits under all possible vertex switchings, which we will call the \emph{labeled switching classes} (or \emph{LSCs}) of $A$. We emphasize the adjective ``labeled'' because the term \emph{switching class} in \cite{douglas} refers to equivalence classes not only under vertex switchings, but also under different types of vertex permutations.

\begin{figure}[htb]
\begin{center}

\begin{tabular}{cc}
\begin{tikzpicture}[scale=0.15]
\SetVertexNoLabel
\SetUpEdge[labelstyle={draw}]
\Vertex[x=0,y=0]{111}
\Vertex[x=0,y=10]{101}
\Vertex[x=0,y=20]{010}
\Vertex[x=0,y=30]{000}
\Vertex[x=-10,y=10]{110}
\Vertex[x=10,y=10]{011}
\Vertex[x=-10,y=20]{100}
\SetVertexNormal[LineWidth=3]
\Vertex[x=10,y=20]{001}
\SetVertexNormal
\Edge[color=red](100)(101)
\Edge[color=red](000)(001)
\Edge[color=red](010)(011)
\Edge[color=red](110)(111)
\Edge[color=green](000)(100)
\Edge[color=green, style=dashed](001)(101)
\Edge[color=green](010)(110)
\Edge[color=green, style=dashed](011)(111)
\Edge[color=blue](000)(010)
\Edge[color=blue, style=dashed](001)(011)
\Edge[color=blue, style=dashed](100)(110)
\Edge[color=blue](101)(111)
\end{tikzpicture}
& 
\begin{tikzpicture}[scale=0.15]
%\SetVertexNormal
\SetVertexNoLabel
\SetUpEdge[labelstyle={draw}]
\Vertex[x=0,y=0]{111}
\Vertex[x=0,y=10]{101}
\Vertex[x=0,y=20]{010}
\Vertex[x=0,y=30]{000}
\Vertex[x=-10,y=10]{110}
\Vertex[x=10,y=10]{011}
\Vertex[x=-10,y=20]{100}
\SetVertexNormal[LineWidth=3]
\Vertex[x=10,y=20]{001}
\Edge[color=red](100)(101)
\Edge[color=red, style=dashed](000)(001)
\Edge[color=red](010)(011)
\Edge[color=red](110)(111)
\Edge[color=green](000)(100)
\Edge[color=green](001)(101)
\Edge[color=green](010)(110)
\Edge[color=green, style=dashed](011)(111)
\Edge[color=blue](000)(010)
\Edge[color=blue](001)(011)
\Edge[color=blue, style=dashed](100)(110)
\Edge[color=blue](101)(111)
\end{tikzpicture}
\end{tabular}
\caption{Before and after a vertex switch at the outlined vertex.\label{fig:switching}}
\end{center}
\end{figure}
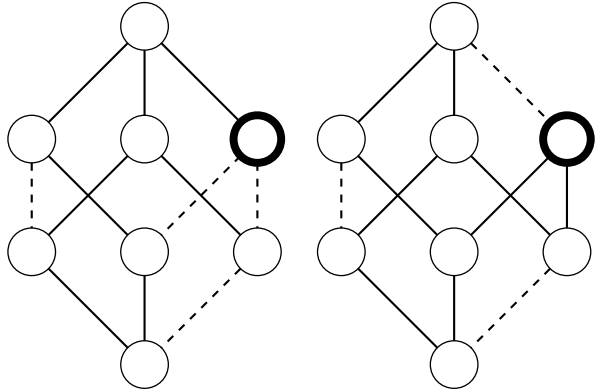

In the representation theory interpretation of adinkras (see Section~\ref{sec:physical motivation}), a vertex switch corresponds to adding a negative sign in front of a component field, which gives an isomorphic\footnote{We talk more about isomorphisms in \ref{sec:clifford representations}. In fact, it is precisely because vertex permutations also give isomorphic representations under this definition that \cite{douglas} uses coarser equivalence classes than we do here.} representation. Thus, it is natural to think about equivalence classes under these transformations. Furthermore, studying vertex switchings will also enable us to better understand the enumeration of dashings (beyond those of $I_c^n$, which we now understand very well). The following computation will not only be useful to study switchings, but will also justify our hunch about the ``degrees of freedom'' from Theorem~\ref{thm:cube dashings}.

\begin{prop}
\label{prop:switching freedom}
In an adinkraizable $(n,k)$-chromotopology $A$, there are exactly $2^{2^{n-k}-1}$ dashings in each LSC.
\end{prop}
\begin{proof}
Vertex switches commute and each vertex switch is an order-$2$ operation, so they form a $\ZZ_2$-vector space, which we may index by subsets of the vertices. Consider a set of vertex switches that fix a dashing. Then, each edge must have its two vertices both switched or both non-switched. This decision can only be made consistently over all vertices if all vertices are switched or all vertices are non-switched. Thus, the $2^{n-k}$ sets of vertex switches  generate a $\ZZ_2$-vector space of dimension $2^{n-k}-1$. This proves the result.
\end{proof}

\begin{cor}
The cubical chromotopology $I^n_c$ has exactly one labeled switching class.
\end{cor}
\begin{proof}
This is immediate from Proposition~\ref{prop:switching freedom} and Theorem~\ref{thm:cube dashings}, with the substitution $k = 0$. Alternatively, this is also evident from \cite[Lemma 4.1]{douglas}.
\end{proof}

\subsection{A Homological Computation}
\label{sec:homology}

Finally, we combinine several ideas (even dashings, vertex switchings, and our cell complex interpretation of chromotopologies) to generalize Theorem~\ref{thm:cube dashings}.

%Theorem~\ref{thm:cube dashings} is equivalent to the following statement: The number of ways to $\ZZ_2$-label $E(I^n)$ such that every $4$-cycle has sum $0$ is $2^{2^n-1}$. A labeling of edges with elements from a ring suggests chains in a cell complex. Furthermore, we can think of the $4$-cycles as boundaries of selected $2$-cells in the said complex. Guided by these thoughts, we can 

%This leads us to the main result of this section.

\begin{prop}
\label{prop:labeled switching classes}
Let $A$ be an adinkraizable $(n,k)$-chromotopology. Then there are $2^k$ LSCs in $A$. 
\end{prop}
\begin{proof}
First, vertex switchings preserve parity of all $4$-cycles, so counting orbits of odd dashings (LSCs) under vertex switchings is equivalent to counting orbits of even dashings.

An even dashing can also be thought of as a formal sum of edges over $\ZZ_2$ (we dash an edge if the coefficient is $1$ and do not otherwise), which is precisely a $1$-chain of $X(A)$ over $\ZZ_2$. Second, the even dashings are defined as dashings where all $2$-colored $4$-cycles have sum $0$. Since these $4$-cycles, as elements of $C_1$, are exactly the boundaries of $C_2$, the even dashings are exactly the orthogonal complement of $\Img(d_2)$ inside of $C_1$ by the usual inner product. Thus, the even dashings have $\ZZ_2$-dimension:
\begin{align*}
\dim((\Img(d_2)^\perp) & = \dim(C_1) - \dim(\Img(d_2)) \\
& = (\dim(\ker(d_1))  + \dim(\Img(d_1))) - \dim(\Img(d_2))\\
& = \dim(H_1) + \dim(\Img(d_1)) \\
& = \dim(H_1) + \dim(C_0) - \dim(H_0). 
\end{align*}

However, note that $\dim(C_0) - \dim(H_0) = 2^{n-k}-1$, which is exactly the dimension of the vector space of the vertex switchings for a particular LSC from Proposition~\ref{prop:switching freedom}. Since the product of the number of LSCs and the number of vertex switchings per LSC is the total number of even dashings, dividing the number of even dashings by $2^{2^{n-k}-1}$ gives that the dimension of switching classes is precisely $\dim(H_1)$.

By Proposition~\ref{prop:covering space} and basic properties of universal covers and fundamental groups, $\pi_1(X(A)) = L$, the quotient group, which in this case is the vector space $\ZZ_2^k$. Also, $H_1 = \ZZ_2^k$ since $H_1$ is the abelianization of $\pi_1$. Thus, we have $2^k$ switching classes.
\end{proof}

Propositions~\ref{prop:labeled switching classes} and \ref{prop:switching freedom} immediately give:

\begin{theorem}
\label{thm:general dashings}
The number of even (or odd) dashings of an adinkraizable $(n,k)$-chromotopology $A$ is  
\[
|e(A)| = |o(A)| = 2^{2^{n-k}+k-1}.
\]
\end{theorem}

%We hope that the original ideas in our approach (the definition of even dashings, the ``divide-and-conquer'' method using decompositions, and the construction of the complex $X(A)$) will offer new perspectives to the subject.
%\input{rankings}

\section{Ranking}
\label{sec:ranking}
 
In Section~\ref{sec:dashing}, we looked at the set of dashings we can put on a chromotopology to make it dashed. In this section, we look at the set of rank functions we can put on a chromotopology\footnote{All results in this section hold for bipartite prechromotopologies since the only requirement we have is bipartiteness. However, we will keep it simple and just talk about chromotopologies.} to make it ranked. Fix a chromotopology $A$. Call the set of all ranked chromotopologies with the same chromotopology as $A$ the \emph{rank family} $R(A)$ and the elements of $R(A)$ \emph{rankings} of $A$. Figure~\ref{fig:2cube rank family} shows the rank family of $I^2$.

\begin{figure}[htb]
\begin{center}

\begin{tabular}{ccc}
\begin{tikzpicture}[scale=0.10]
%\SetVertexNormal
\SetUpEdge[labelstyle={draw}]
\Vertex[x=0,y=0]{00}
\Vertex[x=10,y=10]{10}
\Vertex[x=-10,y=10]{01}
\Vertex[x=0,y=20]{11}
\Edge[color=red](10)(11)
\Edge[color=red](00)(01)
\Edge[color=green](00)(10)
\Edge[color=green](01)(11)
\end{tikzpicture} 
&
\begin{tikzpicture}[scale=0.10]
%\SetVertexNormal
\SetUpEdge[labelstyle={draw}]
\Vertex[x=0,y=0]{10}
\Vertex[x=10,y=10]{00}
\Vertex[x=-10,y=10]{11}
\Vertex[x=0,y=20]{01}
\Edge[color=red](10)(11)
\Edge[color=red](00)(01)
\Edge[color=green](00)(10)
\Edge[color=green](01)(11)
\end{tikzpicture} 
&
\begin{tikzpicture}[scale=0.10]
%\SetVertexNormal
\SetUpEdge[labelstyle={draw}]
\Vertex[x=0,y=0]{11}
\Vertex[x=10,y=10]{01}
\Vertex[x=-10,y=10]{10}
\Vertex[x=0,y=20]{00}
\Edge[color=red](10)(11)
\Edge[color=red](00)(01)
\Edge[color=green](00)(10)
\Edge[color=green](01)(11)
\end{tikzpicture} 
\\
\begin{tikzpicture}[scale=0.10]
%\SetVertexNormal
\SetUpEdge[labelstyle={draw}]
\Vertex[x=0,y=0]{01}
\Vertex[x=10,y=10]{11}
\Vertex[x=-10,y=10]{00}
\Vertex[x=0,y=20]{10}
\Edge[color=red](10)(11)
\Edge[color=red](00)(01)
\Edge[color=green](00)(10)
\Edge[color=green](01)(11)
\end{tikzpicture} 
&
\begin{tikzpicture}[scale=0.15]
%\SetVertexNormal
\SetUpEdge[labelstyle={draw}]
\Vertex[x=0,y=0]{01}
\Vertex[x=0,y=10]{11}
\Vertex[x=10,y=10]{00}
\Vertex[x=10,y=0]{10}
\Edge[color=red](10)(11)
\Edge[color=red](00)(01)
\Edge[color=green](00)(10)
\Edge[color=green](01)(11)
\end{tikzpicture} 
&
\begin{tikzpicture}[scale=0.15]
%\SetVertexNormal
\SetUpEdge[labelstyle={draw}]
\Vertex[x=0,y=0]{00}
\Vertex[x=0,y=10]{10}
\Vertex[x=10,y=10]{01}
\Vertex[x=10,y=0]{11}
\Edge[color=red](10)(11)
\Edge[color=red](00)(01)
\Edge[color=green](00)(10)
\Edge[color=green](01)(11)
\end{tikzpicture} 

\end{tabular}
\caption{The rank family of $I^2$.\label{fig:2cube rank family}}
\end{center}
\end{figure}
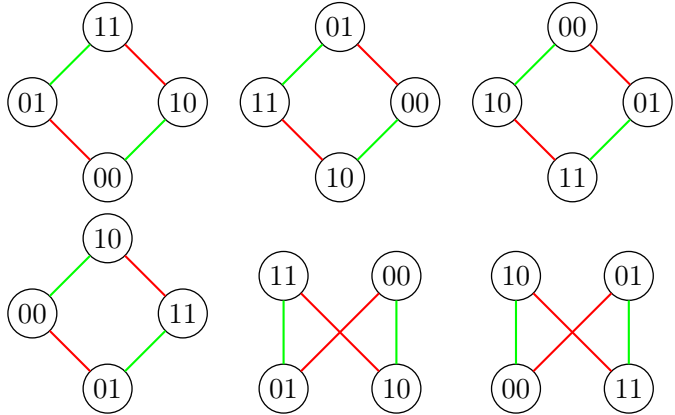

\begin{open}
\label{que:rank family}
What are the enumerative and algebraic properties of $R(A)$? 
\end{open}

After a short survey of the literature (mostly developed in \cite{d2l:graph-theoretic}), we give some original results using the language of posets and lattices in Section~\ref{sec:revisiting}. Finally, with the help of decomposition, we will computationally enumerate the possible rankings for $I^n_c$ with $n \leq 5$.

%Before adding a rank function, adinkras are just dashed chromotopologies. 

%Also, {\cite[Section 11.7]{huffman}} constructs a poset out of cosets of (not necessarily doubly-even) codes that is not directly relevant to our setting

\subsection{Hanging Gardens}

The main structural theorem for rankings is the following theorem. Let $D(v,w)$ be the graph distance between $v$ and $w$.

\begin{theorem} [DFGHIL, {\cite[Theorem 4.1]{d2l:graph-theoretic}}]
\label{thm:hanging gardens}
Fix a chromotopology $A$. Let $S \subset V(A)$ and $h_S\colon S \to \ZZ$ satisfy the following properties:
\begin{enumerate}
\item $h_S$ takes only odd values on bosons and only even values on fermions, or vice-versa.
\item For every distinct $s_1$ and $s_2$ in $S$, we have $D(s_1, s_2) \geq |h_s(s_1) - h_s(s_2)|$.
\end{enumerate}
Then, there exists a unique ranking of $A$, corresponding to the rank function $h$, such that $h$ agrees with $h_S$ on $S$ and $A$'s set of sinks is exactly $S$. By symmetry, there also exists a unique ranking of $A$ whose set of sources is exactly $S$.
\end{theorem}

In other words, any ranking of $A$ is determined by a set of sinks (or sources) and the relative ranks of those sinks/sources. We can think of such a choice as the following: pick some nodes as sinks\footnote{If we chose sources instead of sinks, we can imagine the other nodes ``floating'' up; the name ``Floating Gardens'' also evokes a pleasant image.} and ``pin'' them at acceptable relative ranks, and let the other nodes naturally ``hang'' down. Thus, Theorem~\ref{thm:hanging gardens} is also called the ``Hanging Gardens'' Theorem. Figure~\ref{fig:3cube hanging} shows an example.

\begin{figure}[htb]
\begin{center}

\begin{tabular}{cc}
\begin{tikzpicture}[scale=0.15]
%\SetVertexNormal
\SetUpEdge[labelstyle={draw}]
\SetVertexNoLabel
\Vertex[x=0,y=0]{111}
\Vertex[x=0,y=10]{101}
\Vertex[x=0,y=20]{010}
\Vertex[x=0,y=30]{000}
\Vertex[x=-10,y=10]{110}
\Vertex[x=10,y=10]{011}
\SetVertexNormal[LineWidth=4pt]
\Vertex[x=-10,y=20]{100}
\Vertex[x=10,y=20]{001}
\Edge[color=red](100)(101)
\Edge[color=red](000)(001)
\Edge[color=red](010)(011)
\Edge[color=red](110)(111)
\Edge[color=green](000)(100)
\Edge[color=green](001)(101)
\Edge[color=green](010)(110)
\Edge[color=green](011)(111)
\Edge[color=blue](000)(010)
\Edge[color=blue](001)(011)
\Edge[color=blue](100)(110)
\Edge[color=blue](101)(111)
\end{tikzpicture}
&
\begin{tikzpicture}[scale=0.15]
%\SetVertexNormal
\SetVertexNoLabel
\SetUpEdge[labelstyle={draw}]
\Vertex[x=0,y=0]{111}
\Vertex[x=0,y=10]{101}
\Vertex[x=10,y=0]{010}
\Vertex[x=20,y=10]{000}
\Vertex[x=-10,y=10]{110}
\Vertex[x=10,y=10]{011}
\SetVertexNormal[LineWidth=4pt]
\Vertex[x=-10,y=20]{100}
\Vertex[x=10,y=20]{001}
\Edge[color=red](100)(101)
\Edge[color=red](000)(001)
\Edge[color=red](010)(011)
\Edge[color=red](110)(111)
\Edge[color=green](000)(100)
\Edge[color=green](001)(101)
\Edge[color=green](010)(110)
\Edge[color=green](011)(111)
\Edge[color=blue](000)(010)
\Edge[color=blue](001)(011)
\Edge[color=blue](100)(110)
\Edge[color=blue](101)(111)
\end{tikzpicture}
\end{tabular}

\caption{Left: $I^3$. Right: Hanging Gardens on $I^3$ applied to the two outlined vertices. 
\label{fig:3cube hanging}}
\end{center}
\end{figure}
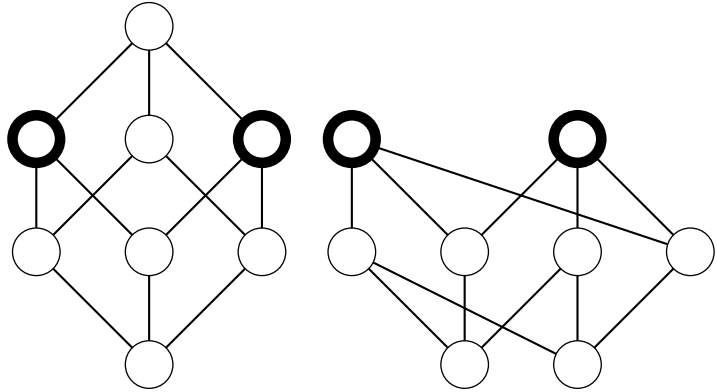

In particular, note that we can pick the set of sinks to contain only a single element, which defines a unique ranking. Thus, for any vertex $v$ of a chromotopology $A$, by Theorem~\ref{thm:hanging gardens} we can get a ranking $A^v$ defined by its only having one sink $v$ (visually, $A^v$ ``hangs'' from its only sink $v$). We call $A^v$ the \emph{$v$-hooked ranking} and all such rankings \emph{one-hooked}. By symmetry, we can also define the \emph{$v$-anchored} ranking $A_v$, which ``floats'' from its only source $v$. For example, Figure~\ref{fig:3cube} is both the $111$-hooked ranking $A^{111}$ and the $000$-anchored ranking $A_{000}$ of $I^3$.

Now, we introduce two operators on $R(A)$.
Given a ranking $B$ in $R(A)$ (with rank function $h$) and a sink $s$, we define $D_s$, the \emph{vertex lowering on $s$}, to be the operation that sends $B$ to the ranking with rank function $h'$ where everything is unchanged except $h'(s) = h(s) - 2$ (visually, we have ``flipped'' $s$ down two ranks and its edges with it). Observe that since $s$ was a sink, this operation retains the fact that all covering relations have rank difference $1$ and thus we still get a ranking. We define $U_s$, the \emph{vertex raising on $s$}, to be the analogous operation for $s$ a source. We call both of these operators \emph{vertex flipping} operators.

\begin{prop} [{DFGHIL, \cite[Theorem 5.1, Corollary 5.2]{d2l:graph-theoretic}}] 
\label{prop:flipping around}
Let $A$ be a ranking. For any vertex $v$:
\begin{enumerate} 
\item there is a sequence of vertex lowerings that take $A$ to $A^v$;
\item there is a sequence of vertex raisings that take $A^v$ to $A$;
\item there is a sequence of vertex raisings that take $A$ to $A_v$;
\item there is a sequence of vertex lowerings that take $A_v$ to $A$.
\end{enumerate}
Furthermore, in all of these sequences we do not need to ever raise or lower $v$.
\end{prop}
\begin{proof}
The main idea of the proof is again very visually intuitive:  ``pin'' $v$ to a fixed rank and let everything else fall down by gravity (slightly more formal: greedily make arbitrary vertex lowerings, except on $v$, until it is no longer possible). The other claims follow by reversing the steps and/or applying symmetry between sinks and sources.
\end{proof}

\begin{cor}
\label{cor:any to any}
Any two rankings with the same chromotopology $A$ can be obtained from each other via a sequence of vertex-raising or vertex-lowering operations.
\end{cor}

Corollary~\ref{cor:any to any} shows that there exists a connected graph $G$ with $V(G) = R(A)$ and $E(G)$ corresponding to vertex flips. In the literature (say \cite{d2l:graph-theoretic}), $R(A)$, equipped with this graph structure, is called the \emph{main sequence}. 

\subsection{The Rank Family Poset}

Consider a chromotopology $A$. We know from the discussion in the previous section that its rank family has the structure of a graph. In this section, we show that it actually has much more structure. Our main goal will be to explicitly prove some observations made in \cite[Section 8]{d2l:graph-theoretic}.

\begin{figure}[htb]
\begin{center}
\includegraphics[scale=0.75]{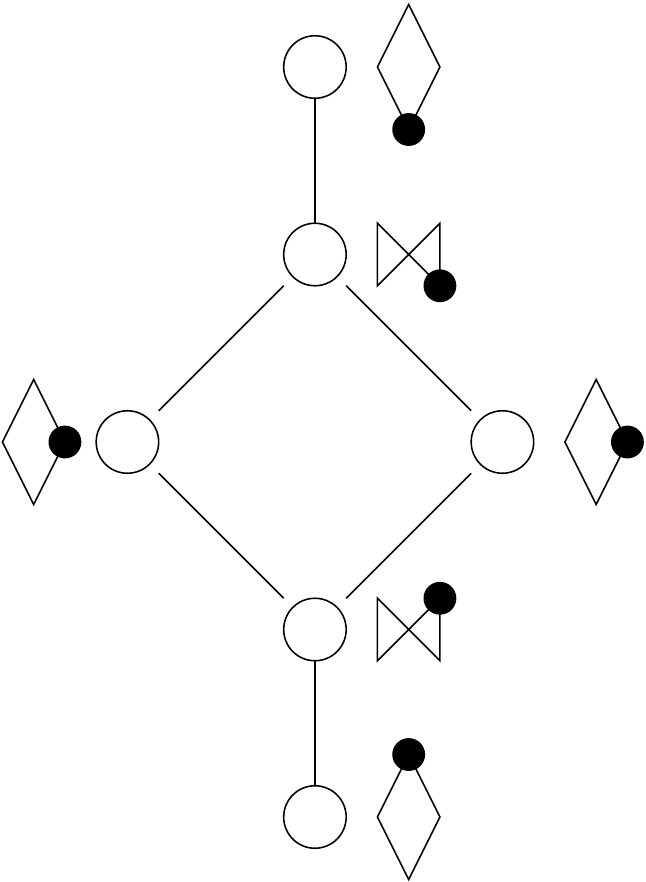}
\hspace{0.5in}
\parbox{5in}{\caption{The rank family poset for $P_v(I^2)$, where next to each node is a corresponding ranking. The rankings are presented as miniature posets, with the black dots corresponding to $v$, the vertex we are not allowed to raise.}
\label{fig:2cube rank poset}}
\end{center}
\end{figure}

\begin{theorem}
\label{thm:generating poset}
For a chromotopology $A$ and any vertex $v$ of $A$, there exists a poset $P_v(A)$ such that:
\begin{enumerate}
\item $R(A)$ is the vertex set of $P_v(A)$;
\item $P_v(A)$ is a symmetric ranked poset, with exactly one element in the top rank and exactly one element in the bottom rank;
\item each covering relation in $P_v(A)$ corresponds to vertex-flipping on some vertex $w \neq v$.
\end{enumerate}
\end{theorem}

\begin{proof}
Construct $P_v(A)$, as a ranked poset, in the following way: on the bottom rank $0$ put $A^v$ as the unique element. Once we finish constructing rank $i$, from any element $B$ on rank $i$, perform a vertex-raising on all sources (except $v$) to obtain a set of rankings $S(B)$. Put the union of all $S(B)$ (as $B$ ranges through the elements on rank $i$) on rank $i+1$, adding covering relations $C > B$ if we obtained $C$ from $B$ via a vertex-raising. It is obvious from this construction that the covering relations in $P_v(A)$ are exactly the vertex flippings on vertices that are not $v$. 

We stop this process if all the elements of rank $i$ have no sources besides $v$ to raise. By Theorem~\ref{thm:hanging gardens}, this is only possible for a single ranking, namely $A_v$. Thus, there is exactly one element in the top rank of $P_v(A)$ as well. By Proposition~\ref{prop:flipping around}, we can get from $A^v$ to any element of $R(A)$ by vertex-raising only, without ever raising $v$. This means that every element of $R(A)$ has appeared exactly once in our construction. 

Now, consider the map $\phi$ that takes a ranking $B$ of rank $k$ to the ranking $B'$, in which each any $v$ with rank $i$ in $B$ has rank $k-i$ in $B'$. Since $\phi$ takes $A_v$ to $A^v$, and vice versa, the top and bottom ranks are symmetric. However, $\phi$ also switches covering relations of vertex-raisings to vertex-lowerings. Thus, we can show that for every $i$ the $i$-th ranks and the $(k-i)$-th ranks are symmetric by induction on $i$. This makes $P_v(A)$ into a symmetric poset as desired.
\end{proof}

Note that the vertex-flips $D_s$ and $U_s$ can be extended linearly to act on formal sums of $\RR[R(A)]$, if we let them send rankings for which the corresponding flip is not allowed to $0$. Define $U(A)$ to be the algebra generated by all $U_s$ with $s \in A$. 

\begin{cor}
\label{cor:verma} 
The image of $A^v$ under the action of the quotient $U(A)/U_v$ is $\RR[R(A)].$
\end{cor}
\begin{proof}
This is immediate from the construction in Theorem~\ref{thm:generating poset}, where we started with a single ranking $A^v$. Taking the image under vertex raisings is exactly taking the image of the $U(A)$-action on $A^v$. Forbidding the vertex raising at $v$ is exactly restricting this action to the quotient $U(A)/U_v$. 
\end{proof}

\begin{figure}[htb]
\begin{center}
\includegraphics[scale=0.75]{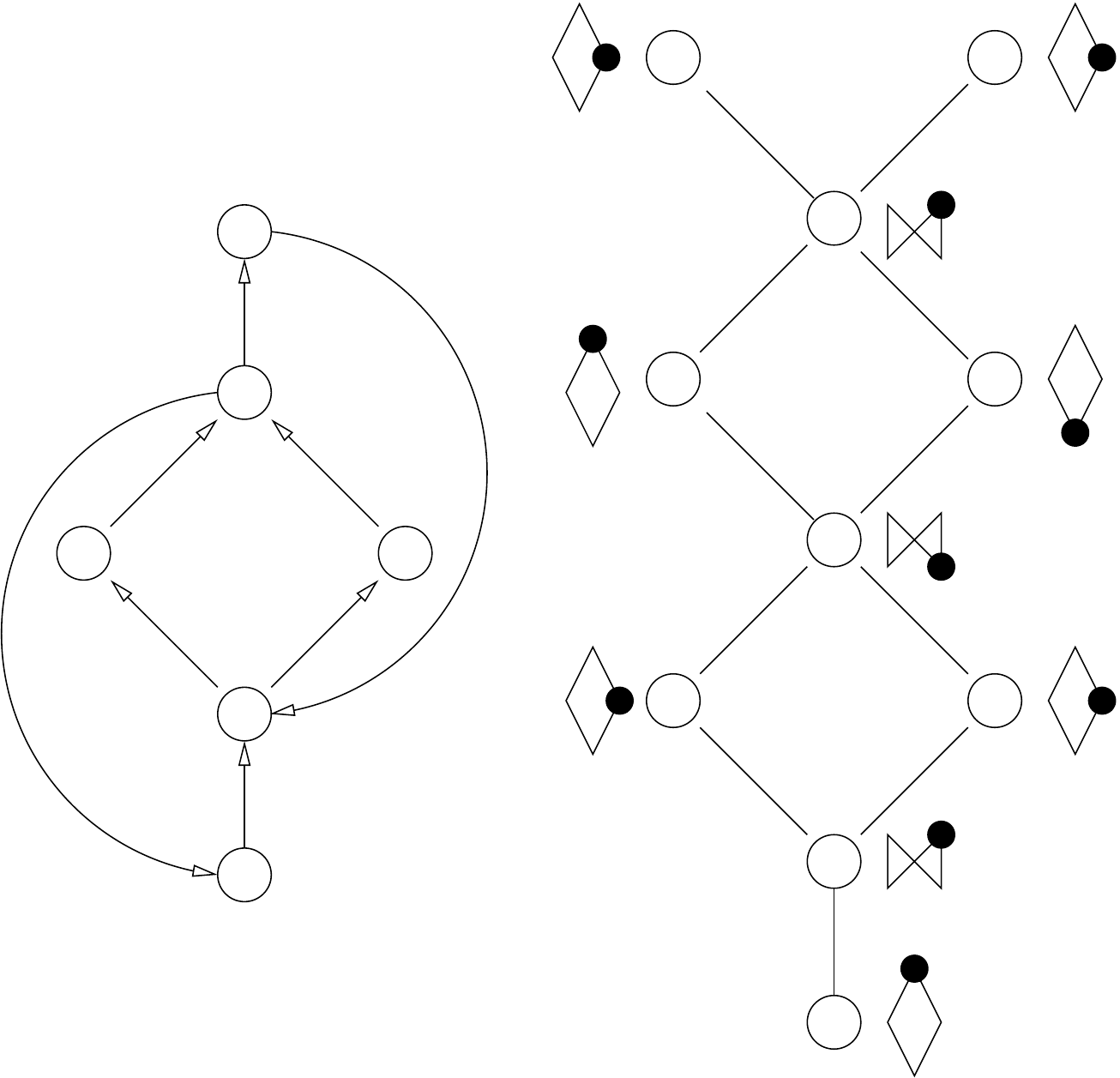}
\hspace{0.5in}
\parbox{5in}{\caption{Left: if we were to allow vertex raising at $v$, we no longer get a poset since we introduce cycles. Right: we can also think of this situation by an infinite poset leading upwards.}
\label{fig:2cube verma}}
\end{center}
\end{figure}

The authors of \cite{d2l:graph-theoretic} noted that the rank family is reminiscent of a Verma module. Corollary~\ref{cor:verma} is an algebraic realization of this observation. The ranking $A^v$ takes the role of the \emph{lowest-weight vector}. If we allowed vertex raisings at $v$, we would have obtained an infinite repeating family of rankings, as in Figure~\ref{fig:2cube verma}. When we strip the redundant rankings by forbidding $U_v$, we leave ourselves with a finite $R(A)$. 

\subsection{Revisiting the Hanging Gardens}
\label{sec:revisiting}

In this section, we will put even more structure on $P_v(A)$ with the language of \emph{lattices}. A quick overview of the concepts we will need are in Appendix~\ref{app:lattices}.

We first construct an auxiliary poset $E_v(A)$, which we call the \emph{$v$-elevation poset} of $A$:
\begin{itemize}
\item let the vertices of $E_v(A)$ be ordered pairs $(w, h)$, where $w \in A$, $w \neq v$, and $h \in \{1, 2, \ldots, D(w,v)\}$.
\item whenever $D(w_1, w_2) = 1$ and $D(w_1, v) + 1 = D(w_2, v)$, have $(w_1, h)$ cover $(w_2, h)$ and $(w_2, h+1)$ cover $(w_1, h)$.
\end{itemize}

\begin{figure}[htb]
\begin{center}

\begin{tikzpicture}[scale=0.15]
%\SetVertexNormal
\SetUpEdge[labelstyle={draw}]
\SetVertexNoLabel
% \Vertex[x=-30,y=0]{X}
% \Vertex[x=-20,y=0]{Y}
% \Vertex[x=-10,y=10]{Z}
% \Vertex[x=-10,y=-10]{W}
% \Edges(X, Z, Y, W, X)
% \draw (-30, -30) node []{$10$};
% \draw (20, 10) node []{$h = 1$};
% \draw (20, 20) node []{$h = 2$};
% \draw (20, 30) node []{$h = 3$};
\Vertex[x=0,y=0]{A}
\Vertex[x=10,y=0]{B}
\Vertex[x=20,y=0]{C}
\Vertex[x=30,y=10]{D}
\Vertex[x=40,y=10]{E}
\Vertex[x=50,y=10]{F}
\Vertex[x=30,y=-10]{G}
\Vertex[x=40,y=-10]{H}
\Vertex[x=50,y=-10]{I}
\Vertex[x=60,y=20]{J}
\Vertex[x=60,y=0]{K}
\Vertex[x=60,y=-20]{L}
% \draw (0, 25) node []{100};
% \draw (10, 25) node []{010};
% \draw (20, 25) node []{001};
% \draw (30, 25) node []{110};
% \draw (40, 25) node []{101};
% \draw (50, 25) node []{011};
% \draw (60, 25) node []{111};

% \draw (20, 10) node []{$h = 1$};
% \draw (20, 20) node []{$h = 2$};
% \draw (20, 30) node []{$h = 3$};

\foreach \w in {D,E,F,G,H,I}{%
  \Edge(K)(\w)};%
\foreach \w in {D,E,F}{%
  \Edge(J)(\w)};%
\foreach \w in {G,H,I}{%
  \Edge(L)(\w)};%
\Edges(A,D,B,F,C,E,A,G,B,I,C,H,A)

\end{tikzpicture}
\caption{The elevation poset $E_{000}(I_c^3)$. Nodes $(w,h)$ on each vertical line have the same $w$-value. In order (left to right), the $w$-values are: $100, 010, 001, 110, 101, 011, 111$.
\label{fig:elevation posets}}
\end{center}
\end{figure}
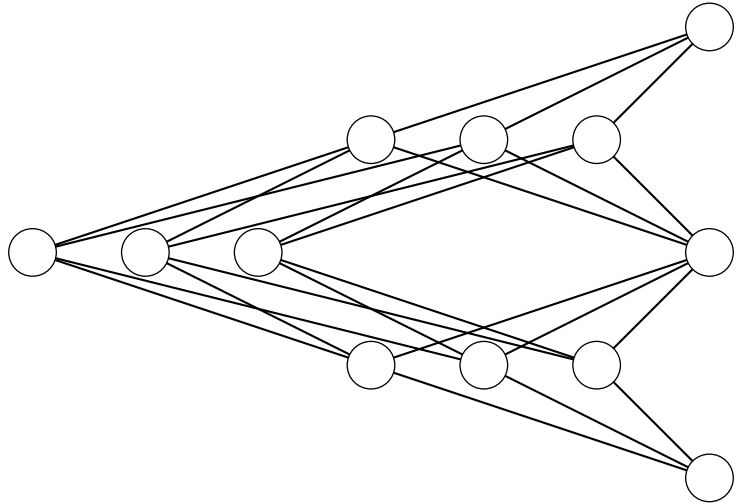

Our main result of this section is the following theorem.

\begin{theorem}
\label{thm:elevation poset}
The $v$-elevation poset and the $v$-rank family poset are related by
\[
P_v(A) = J(E_v(A)).
\]
Thus, the rank family poset $P_v(A)$ is a finite distributive lattice.
\end{theorem}
\begin{proof}
We show that there is a bijection between order ideals of $E_v(A)$ and elements of $P_v(A)$. The second claim in the theorem follows immediately from Theorem~\ref{thm:ftfdl}.

%To see this, first note that for any $(w, h)$ and $(w, h+1)$ in $E_v(A)$, we have $(w, h+1) > (w, h)$, since for some neighbor $w'$ of $w$ we have $(w, h+1) > (w', h') > (w, h)$, where $h' = h+1$ if $d(w, v) + 1 = d(w', v)$ and $h' = h$ if $d(w, v) - 1 = d(w', v)$. Thus, every order ideal $I$ of $E_v(A)$ corresponds to picking a function $f$ on all $w \neq v$, such that $f(w) \in \{0, 1, \ldots, d(w,v)\}$. Each $I$ includes all ordered pairs (possibly none) of the form $(w, h)$ where $h \leq f(w)$. Now, there is a trivial bijective extension of such a function $f$ to all $w$ by setting $f(w) = 0$. Each such function corresponds to exactly one ideal $I$ of $E_v(A)$.

Each ideal $I$ of $E_v(A)$ gives exactly one ranking in $P_v(A)$, as follows: for every vertex $w \in A$, take the maximum $h \in \ZZ$ such that $(w, h) \in I$, taking $h = 0$ if no $(w,h)$ appears in $I$. Now assign to $w$ the rank $2h-D(w,v)$. In other words, $h$ indexes the elevation of $w$ by counting the number of total times we flip $w$ up from the initial state of the $v$-hooked ranking (which corresponds to the empty ideal), justifying the name \emph{elevation poset}. It remains to check that this map is a bijection. 

Take a ranking $A'$ in $P_v(A)$. For any $w \in A'$, we have $h(w) = 2h - D(w, v)$ for some $0 \leq h \leq D(w,v)$. Define $I \subset V(E_v(A))$ to contain all $(w,h')$, possibly empty, with $h' \leq h$. The property of $A'$ being a ranking is equivalent to the condition that for every pair of neighbors $w_1$ and $w_2$ with $D(v, w_2) = D(v, w_1)+1$, we have $|h(w_1) - h(w_2)| = 1$. However, this in turn is equivalent to the condition that the maximal $h_1$ and $h_2$ such that $(w_1, h_1)$ and $(w_2, h_2)$ appear (as before, define one of them to be $0$ if no corresponding vertices exist in $I$) in $I$ satisfy either $h_1 = h_2$ or $h_1 = h_2+1$, which is exactly the requirement for $I$ to be an order ideal. Thus, our bijection is complete. 
\end{proof}

The proof of Theorem~\ref{thm:elevation poset} gives another interpretation of Theorem~\ref{thm:hanging gardens}. Consider the order ideals of $E_v(A)$. Each such order ideal corresponds to an antichain of maximal elements, which is some collection of $(w_i, h(w_i))$. It can be easily checked that in the corresponding element of $P_v(A)$, the $w_i$ are exactly the sinks, placed at rank $2h_i - D(w_i, v)$. 

Even though Theorem~\ref{thm:elevation poset} gives us more structure on the rank family, it is very difficult in general to count the order ideals of an arbitrary poset. The typical cautionary tale is the case of the (extremely well-understood) Boolean algebra $B_n$, where the problem of counting the order ideals, known as Dedekind's Problem, has resisted a closed-form solution to this day, with answers computed up to only $n = 8$ (see \cite{wiedemann:dedekind}). 

\subsection{Ranking the Cubical Chromotopology}
\label{sec:counting rankings}

Counting the cardinality of $R(I_c^n)$ for general $n$ seems difficult. Instead, we'll attempt an algorithmic attack with the help of decomposition.
 
For any $A \in R(I_c^n)$, recall from Section~\ref{sec:decomposition} that the color $n$ must decompose $A$ uniquely into $A_0 \nearrow_n A_1$ or $A_0 \searrow_n A_1$, where each of $A_0$ and $A_1$ is a ranking in $R(I_c^{n-1})$. Thus, we can iterate over potential pairs of rankings $(A_0, A_1)$ and see if each of them could have come from some $A$. It suffices to check that each pair of vertices $\inc(c, 0 \to n)$ and $\inc(c, 1 \to n)$, where $c \in \ZZ_2^{n-1}$, has rank functions differing by exactly $1$. However, this requires $2^{n-1}$ comparisons for each pair of ranking in $R(I^{n-1})$. The following lemma cuts down the number of comparisons.

\begin{lemma}
\label{lem:cutting hanging gardens}
For an $(n,k)$-ranking $A$ and $(n-1, k)$-rankings $A_0$ and $A_1$, we have $A = A_0 \nearrow_n A_1$ if and only if the colors and vertex labelings of the three rankings are consistent and the following condition is satisfied: for each $c \in \ZZ_2^{n-1}$ and the pair of corresponding vertices $s_0 = \inc(c, 0 \to n)$ and $s_1 = \inc(c, 1 \to n)$ such that at least one of $s_0$ or $s_1$ is a sink (in $A_0$ or $A_1$, respectively), we have $|h(s_0) - h(s_1)| = 1$.
\end{lemma}
\begin{proof}
These are clearly both necessary conditions. It obviously suffices if the adjacency condition $|h(s_0) - h(s_1)| = 1$ were checked over all $c$ for all pairs $s_0$ and $s_1$ in $A$ corresponding to the same $c$. It remains to show that checking the situations where at least one $s_i$ is a sink (in their respective $A_i$) is enough.

Suppose we had a situation where checking just these pairs were not enough. This means for all pairs of vertices corresponding to the same $c$ (where at least one vertex is a sink in its respective $A_i$) we meet the adjacency condition, but for some such pair where neither are sinks, we have $|h(s_0) - h(s_1)| \neq 1$. Let $(s_0^{(0)} = s_0, s_1^{(0)} = s_1)$ be such a pair. Without loss of generality, assume $h(s_0) > h(s_1)$. Since $s_0$ is not a sink in $A_0$, there is some $s_0^{(1)}$ covering $s_0^{(0)}$ by an edge with some color $i \neq n$. Similarly, let $s_1 = s_1^{(0)}$ be connected to $s_1^{(1)}$ via color $i$. Continuing this process, we eventually must come to a pair of vertices $s_0^{(l)}$ and $s_1^{(l)}$ where at least one is a sink. However, $h(s_0^{(l)}) = h(s_0) + l > h(s_1^{(l)})$. Since we assumed that $h(s_0) > h(s_1)$ and $|h(s_0^{(l)}) - h(s_1^{(l)})| = 1$, the only way for these equations to be possible is if for each $i < l$, we had $h(s_1^{(i+1)}) = h(s_1^{(i)}) + 1$. But this meant $|h(s_0^{(0)}) - h(s_1^{(0)})| = 1$ in the first place, a contradiction. Thus, these situations do not exist, and it suffices to only check pairs where at least one vertex is a sink.
\end{proof}

Lemma~\ref{lem:cutting hanging gardens} makes the following algorithm possible:

% \begin{framed}
% \begin{enumerate}
% \item For the data structure, represent all rankings $A$ by a set of sinks $S(A)$ and their ranks as in Theorem~\ref{thm:hanging gardens}. 
% \item Start with $R(I_c^1)$, which is a set of $2$ rankings.
% \item Given a set rankings in $R(I_c^{n-1})$, construct the rankings in $R(I_c^n)$ in the following manner:
% \begin{enumerate}
% \item Iterate over all pairs of (possibly identical) rankings $(A, B)$ in $R(I_c^{n-1}) \times R(I_c^{n-1})$.
% \item For each pair, consider the ranking $B'$ which is identical to $B$ except with the rank function $h_{B'}(\overrightarrow{0}) = h_B(\overrightarrow{0}) + 1$.
% \item For each sink $s \in S(A) \cup S(B')$, check that $|h_A(s) - h_{B'}(s)| = 1$. 
% \item If the above is satisfied for all $s$, put $A \nearrow_n B'$ in  $R(I_c^n)$.
% \end{enumerate}
% \item Repeat the previous step.
% \end{enumerate}
% \end{framed}

\begin{framed}
\begin{enumerate}
\item For the data structure, represent all rankings $A$ by a set of sinks $S(A)$ and their ranks as in Theorem~\ref{thm:hanging gardens}. 
\item Start with $R(I_c^1)$, which is a set of $2$ rankings.
\item Given a set rankings in $R(I_c^{n-1})$, iterate over all pairs of (possibly identical) rankings $(A, B)$ in $R(I_c^{n-1}) \times R(I_c^{n-1})$. For each pair,
\begin{enumerate}
\item Consider the ranking $B'$ which is identical to $B$ except with the rank function $h_{B'}(\overrightarrow{0}) = h_B(\overrightarrow{0}) + 1$.
\item For each sink $s \in S(A) \cup S(B')$, check that $|h_A(s) - h_{B'}(s)| = 1$. 
\item If the above is satisfied for all $s$, put $A \nearrow_n B'$ in  $R(I_c^n)$.
\end{enumerate}
\end{enumerate}
\end{framed}

We used the above algorithm to compute the results for $n \leq 5$, which we include in Table~\ref{table:counting cubes} along with the counts of dashings (recall this is $o(n) = 2^{2^n-1}$) and adinkras (which we obtain by multiplying $|R(I_c^n)|$ and $o(n)$ as the dashings and rankings are independent). Finding the answer for $n = 6$ seems intractible with an algorithm that is at least linear in the number of solutions. For chromotopologies other than $R(I_c^n)$ that can be decomposed, Lemma~\ref{lem:cutting hanging gardens} still allows us to perform some similar computations. However, doing a case-by-case analysis for different chromotopologies seems uninteresting without unifying rules.

\begin{table}
\begin{center}$
\begin{array}{|l|l|l|l|}
\hline 
n & \text{dashings} & \text{rankings} & \text{adinkras}\\ 
\hline
1 & 2 &  2 & 4 \\
2 & 8 & 6 & 48 \\ 
3 & 128 & 38 & 4864 \\ 
4 & 32768 & 990 & 32440320 \\
5 & 2147483648 & 395094 & 848457904422912 \\
\hline
\end{array}
$\end{center}
\caption{Enumeration of dashings, rankings, and adinkras with chromotopology $I_c^n$. \label{table:counting cubes}}
\end{table}

\section{Back to Physics}
\label{sec:together}

So far, we mostly discussed pure combinatorics in our discussions about chromotopologies, dashings, and rankings. We now revisit the original physics context. We survey the recent papers but also examine some of the foundational questions. In particular, we suggest a definition of isomorphism for adinkraic representations, a notion that has not yet been rigorously treated in the literature.

\subsection{Constructing Representations}
\label{sec:construction}

Take an adinkra $A$, and consider the component fields (the bosons $\phi$ and the fermions $\psi$) as a basis. Then, consider a set of matrix generators $\{\rho(Q_i)\}$ in that basis, where $\rho(Q_i)$ is the adjacency matrix of the subgraph of $A$ induced by the edges of color $i$.  If we order all the $\phi$ to come before all the $\psi$ in the row/column orderings, these matrices are block-antisymmetric of the form
$$\rho(Q_i) = \begin{pmatrix}0 & L_i \\ R_i & 0\end{pmatrix},$$
where the $L_i$ and $R_i$ are \cite{gates:genomics}'s \emph{garden matrices}. For the adinkra in Figure~\ref{fig:2cube}, we have the following matrices, where the row/column indices are in the order $00, 11, 10, 01$.
$$\rho(Q_1) = \begin{pmatrix}0 & 0 & 1 & 0 \\ 0 & 0 & 0 & 1 \\ 1 & 0 & 0 & 0 \\ 0 & 1 & 0 & 0 \end{pmatrix} \hspace{1 cm} \rho(Q_2) = \begin{pmatrix} 0 & 0 & 0 & 1 \\ 0 & 0 & -1 & 0 \\ 0 & -1 & 0 & 0 \\ 1 & 0 & 0 & 0 \end{pmatrix}$$
% 00 11 10 01
\begin{figure}[htb]
\begin{center}

\begin{tikzpicture}[scale=0.10]
%\SetVertexNormal
\SetUpEdge[labelstyle={draw}]
\Vertex[x=0,y=0]{00}
\Vertex[x=10,y=10]{10}
\Vertex[x=-10,y=10]{01}
\Vertex[x=0,y=20]{11}
\Edge[color=red, style=dashed](10)(11)
\Edge[color=red](00)(01)
\Edge[color=green](00)(10)
\Edge[color=green](01)(11)
\end{tikzpicture} 

\caption{An adinkra with chromotopology $I^2$. \label{fig:2cube}}
\end{center}
\end{figure}
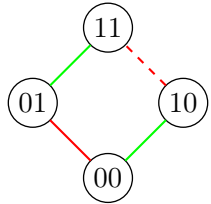

So far, we have encoded the graph and the dashing into the matrices, but we do not yet have a representation of the supersymmetry algebra $\po{1}{N}$. In fact, the $\rho(Q_i)$ form a representation of the Clifford algebra $\cl(N)$, which we discuss further in Section~\ref{sec:clifford representations}. The missing information (up to scalars) is the ranks of the vertices, which we can add into these matrices by adding the Hamiltonian operator $H$'s to appropriate entries (recall Section~\ref{sec:physical motivation} for details). In this sense, we are partitioning the infinite-dimensional basis of the representation into finite-dimensional ``slices,'' each slice corresponding to a single finite-dimensional representation corresponding to our finite-dimensional matrices.

An obvious question to consider whenever we study representations is the following:

\begin{open}
\label{que:irreducible}
Which adinkratic representations are irreducible? 
\end{open}

In the valise case, this is well-understood (see \cite{d2l:clifford}) with a surprising answer. If $L$ were not a maximal subspace inside $\ZZ_2^n$, we may quotient $I_c^n/L$ further to give a subrepresentation. Thus, irreducible valise adinkratic representations must have maximal doubly-even codes, which are self-orthogonal. There seems to be no good general method for other rankings. The intuition of the obstruction is that this method of creating subrepresentations require the vertices in each coset to come from the same rank, corresponding to the same engineering dimension. This kind of physics constraint is intricately connected to the selection of the right notion of isomorphism for adinkratic representations, which we now discuss.

\subsection{When are Two Adinkras Isomorphic?}
\label{sec:isomorphism}

A natural problem in considering representations is to selecting the right definition of isomorphism. The instinct for this choice seems to be completely intuitive for the authors of the literature (see \cite{gates:genomics} and \cite{douglas}), but this may be the first formal discussion.

\begin{open}
\label{que:isomorphism classes}
What is the right definition of ``isomorphism'' for two adinkraic representations? How does it relate to the combinatorics of adinkras?
\end{open}

We usually consider two representations isomorphic if they are conjugate by some change of basis. However, because of our physics context we need more restrictions. To find the right notion, we now recall/define three types of transformations and discuss what it means for them to give the ``same'' adinkra.
\begin{itemize}
\item Recall that a \emph{vertex switching} changes the dashing of all edges adjacent to a vertex. This corresponds to simply changing the sign (as a function) of the component field corresponding to that vertex, or equivalently, conjugation of the representation by a diagonal matrix of all $1$'s except for a single $(-1)$. It is reasonable to consider this move as an operation that preserves isomorphism.
\item Let a \emph{color permutation} permute the names of the colors (in the language of codes, it is a simultaneous column permutation of the bitstrings corresponding to each vertex). In our situation, this is just a shuffling of the generators, so at first glance it is reasonable to consider this operation to preserve isomorphism. The existing literature, e.g. \cite{d2l:clifford}, seems to do so as well. However, this is not quite what we want in a natural definition, where we need to consider the base ring fixed. By analogy, consider the $k[x,y]$-modules $k[x,y]/(x)$ and $k[x,y]/(y)$, which may look  ``equivalent'' (they are indeed isomorphic as algebras) but are not isomorphic as modules. They should not be: we really want $A \oplus B$ to be isomorphic to $A \oplus B'$ if $B$ and $B'$ were isomorphic, but the direct sums $k[x,y]/(x) \oplus k[x,y]/(x)$ and $k[x,y]/(x) \oplus k[x,y]/(y)$ are not isomorphic in any reasonable way. 

In fact, the existing adinkra literature notices this problem when considering disconnected adinkras (i.e. adinkras with topology of a disconnected graph). These graphs correspond to direct sums of representations of single adinkras. However, since a color permutation is done over all disjoint parts simultaneously, if we consider color permutations as operations that preserve isomorphism, we obtain situations where $A \cong C$ and $B \cong D$, but $A \oplus B \not \cong C \oplus D$. The literature deals with this situation by calling color permutations \emph{outer isomorphisms}. We believe the correct thing to do is to just to {\bf not} consider these situations isomorphic and treat them as a separate kind of similarity.

\item Let a \emph{vertex permutation} permute the vertex labels of an adinkra $A$. This corresponds to conjugating the matrices $\rho(Q_i)$ by permutation matrices. Here it makes sense to impose further physics constraints: we want these transformations to preserve the engineering dimensions of the component fields. This prevents us from allowing arbitrary vertex permutations and simply considering two adinkraic representations isomorphic if they're conjugate. On the adinkras side, this corresponds to us enforcing that the rank function of $A$ be preserved under any vertex permutation (in particular, bosons must go to bosons, and fermions to fermions). Happily, this neatly corresponds to the natural definition of isomorphism for ranked posets.
\end{itemize}

When we have a combinatorial representation of a physics situation, there are two ways our brains naturally want to define \emph{isomorphism}, one using the physics intuition (which is the ``right'' one but harder to see), and one using the combinatorial intuition (which is ``wrong'' but easier to see).

\begin{itemize}
\item Following physics requirements, we define two adinkras $A$ and $B$ to be \emph{isomorphic} if there is some matrix $R$ that transforms each generator $\rho(Q_I)$ of $A$ to the corresponding $\rho(Q_I)$ in $B$ via conjugation, with the stipulation that such a conjugation preserves the ranks of the component fields. To be explicit, let the component fields be partitioned into $P_1 \cup P_2 \cup \cdots$, where each $P_i$ contains all $\psi_j$ or $\phi_j$ of some rank (equivalently, engineering dimension). We require $M$ to be block-diagonal with respect to this partition.
\item Following combinatorial intuition, we define two adinkras to be \emph{$C$-isomorphic} if there is a sequence from one to the other via only vertex switchings or ranked poset isomorphisms. Note that from our earlier discussion, both of these operations preserve isomorphism, so $C$-isomorphism is more restrictive than isomorphism\footnote{In turn, our \emph{isomorphism} is more restrictive than a physics constraint put forth in \cite{gates:genomics}, which can be restated as the requirement that $M$ must be block-diagonal on $\Phi \cup \Psi$, the partition into bosons and fermions (this is coarser than the partition into the $P_i$).}. 
\end{itemize}

In a perfect world these two notions would exactly match. We're not so lucky here: the adinkra topology is an invariant of the operations in the definition of $C$-isomorphism, but there are adinkras with different topologies that correspond to isomorphic representations; see \cite[Examples 4.2, 4.5]{d2l:clifford}. We still do not have a complete picture of the nuances between the two definitions.

$C$-isomorphism is studied in more detail in \cite{douglas} (where it is simply called ``isomorphism''), which gives a deterministic algorithm to tell if two adinkras are $C$-isomorphic. Similar discussion relevant to isomorphism (even though it was not defined as such) can be found in \cite{gates:genomics} and \cite{d2l:clifford}. We believe a lot of interesting mathematics remain in the area. As a potential example, both \cite{gates:genomics} and \cite{douglas} distinguish adinkras with the help of what amounts to the trace of the matrix 
$$ \rho(Q_1) \rho(Q_2) \cdots \rho(Q_N) $$
after multiplying by the matrix $\begin{pmatrix} I & 0 \\ 0 & -I \end{pmatrix}. $ We would like to point out that this is precisely the well-known \emph{supertrace} from the theory of superalgebras. 

Finally, we can rephrase our notion of \emph{isomorphism} with the language of quivers. Our isomorphism classes are exactly isomorphism classes of quiver representations with each node corresponding to one of the partitions $P_i$. For example, a valise representation (which corresponds to a $2$-partition) for $I_c^2$ cam be identified with the quiver from Figure~\ref{fig:quiver}, where the $2$ nodes correspond to the bipartition $B \cup F$ of bosons and fermions. Each pair of edges of a color, say $i$, corresponds to the off-diagonal block matrices $M$ and $M'$ that arise when we write each $Q_i$ as $\begin{pmatrix} 0 & M \\ M' & 0 \end{pmatrix}$ corresponding to the bipartition. In general, if $B$ is further partitioned into $m_1$ parts and $F$ into $m_2$ parts, we would obtain $(m_1 + m_2)$ vertices in total and $2m_1 m_2$ edges of each color. 

\begin{figure}[htb]
\begin{center}

\begin{tikzpicture}[scale=0.15]
%\SetVertexNormal
\SetUpEdge[labelstyle={draw}]
\Vertex[x=0,y=0]{B}
\Vertex[x=15,y=0]{F}
\Edge[style={->, bend left=10, color=red}](B)(F)
\Edge[style={<-, bend right=10, color=red}](B)(F)
\Edge[style={->, bend left=40, color=green}](B)(F)
\Edge[style={<-, bend right=40, color=green}](B)(F)
%\Loop[dist=10cm, dir=WE](B)
%\Loop[dist=10cm, dir=EA](F)
\end{tikzpicture} 

\caption{The quiver corresponding to a valise adinkra for $I_c^2$. \label{fig:quiver}}
\end{center}
\end{figure}
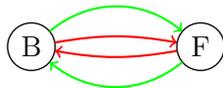

\subsection{Clifford Representations}
\label{sec:clifford representations}

In Section~\ref{sec:physical motivation}, we called the adinkraic representations arising from valise adinkras ``Clifford supermultiplets.'' This is no big surprise -- when we ignore the Hamiltonian $H$ in the defining relations 
$$\{Q_I, Q_J\} = 2\delta_{IJ}H,$$ 
we get precisely the Clifford algebra relations 
$$\{Q_I, Q_J\} = 2\delta_{IJ}.$$ 

In other words, when we forget about the rank of an adinkra and look at only the dashed chromotopology (alternatively, the valise, where no bosons or fermions are privileged by rank from the other fields of the same type), we are really looking at a Clifford algebra representation, something that we saw in \ref{sec:construction} and in the proof of Theorem~\ref{thm:adinkraizable as quotient}. Therefore, we can think of adinkraic representations as extensions of representations of the Clifford algebra, a well-known subject (see \cite{atiyah} or \cite{lawson}). \cite{d2l:supermodules} makes this analogy more rigorous by realizing these representations as filtered Clifford supermodules. 

While Clifford algebras are well-understood, the following is a natural question to ask:
\begin{open}
Can adinkras give us better intuition (organizational or computational) about the theory of Clifford representations?
\end{open}

In \cite{d2l:clifford}, each valise adinkra with the $I^n_c$ chromotopology is used to explicitly construct a representation of the Clifford algebra $\cl(n)$. This introduces a plethora of representations with lots of isomorphisms between them -- after all, there are at most $2$ irreducible representations for each Clifford algebra over $\RR$.

\subsection{Extensions}
\label{sec:extensions}
As we brushed over in Section~\ref{sec:physical motivation}, the adinkraic representations correspond to the $1$-dimensional (more precisely, $(1,0)$-dimensional) worldline situation with $\N{}$ supercharge generators. We talk about the more general context in this section. Helpful expositions of related concepts are \cite{freed} and \cite{varadarajan}.

In general, we are interested in some $(1+q)$-dimensional vector space over $\RR$ with Lorentzian signature $(1,q)$. Besides our $(1,0)$ situation, some examples are $(1,1)$ (\emph{worldsheet}) and $(1,3)$ (\emph{Lorentzian spacetime}). We can write this more general situation as $(1,q|\N{})$-supersymmetry\footnote{Here is another unfortunate source of language confusion: for physicists, $\N{}$ means the number of supersymmetry generators, whereas mathematicians would instead count the total number of dimensions and write $d\N{}$ instead of $\N{}$, where $d$ is the real dimension of the minimal spin-($1/2$) representation of $\RR^{1+q}$. These minimal dimensions are $1,1,2,4,8,\ldots$ starting with $n=1$. Luckily, for most of this paper, $d=1$ and we have no problems. A clear explanation is given in \cite{freed}.} We would then call the corresponding superalgebra $\po{1+q}{\N}$, which specializes to the particular superalgebra $\po{1}{\N{}}$ we have been working with when $q = 0$.

In the case where $(1+q) = 2, 6 \pmod{8}$, we actually get two different types of supercharge generators (this again corresponds to the fact that there are two Clifford algebra representations over $\RR$ in those situations), so we can partition $\N = P+Q$ and call these situations $(1,q|P,Q)$-supersymmetry.

\begin{open}
What happens when we look at $q>0$? What kind of combinatorial objects appear? Will the machinery we developed for adinkraic representations in the wordline case be useful? 
\end{open}

\cite{gates:dimensional_extension} examines the $(1,1)$-case, where the combinatorics get more complex. The reader may have gotten the intuition that the dashings and rankings are fairly independent conditions of the adinkra. This is true for the $(1,0)$-case but no longer holds for the $(1,1)$-case, where certain forbidden patterns arise that depend both on the dashings and rankings. \cite{hubsch:weaving} creates $(1,1|P,Q)$ representations by tensoring and quotienting worldline representations, similar in spirit to the construction of representations of semisimple Lie algebras.

In a different direction, \cite{faux:dimensional_enhancement} and \cite{faux:spin_holography} examine which worldline representations can be ``shadows'' of higher-dimensional ones and give related consistency-tests and algorithms. As worldline representations are involved, the $1$-dimensional theory already built plays an instrumental role.

\section{Acknowledgements}
First and foremost, we thank Sylvester ``Jim'' Gates for teaching the subject to us. We thank Brendan Douglas, Greg Landweber, Kevin Iga, Richard Stanley, Joel Lewis, and Steven Sam for helpful discussions. Anatoly Preygel and Nick Rozenblyum gave us very enlightening lessons in algebraic topology and made the relevant sections possible. Alexander Postnikov made the observation about quivers. We are especially grateful to Tristan H\"{u}bsch for his unreasonably generous gift of time and patience through countless communications. 

The author was supported by an NSF graduate research fellowship.

\appendix

\section{Clifford Algebras}
\label{app:clifford}

The \emph{Clifford algebra} is an algebra $\cl(n)$ with generators $\gamma_1, \ldots, \gamma_n$ and the anticommutation relations 
\[
\{\gamma_i, \gamma_j\} = 2\delta_{i,j} \cdot 1.
\]

The Clifford algebra can be defined for any field, but we will typically assume $\RR$. There are also more general definitions than what we give, though we won't need them for our paper. For references, see \cite{atiyah} or \cite{deligne}.

We can associate an element of the Clifford algebra to any $n$-bitstring $b = b_1 b_2 \cdots b_n$, by defining
\[
\clif(b) = \prod_i \gamma_i^{b_i},
\]
where the product is taken in increasing order of $i$. Call these elements \emph{monomials}. The $2^n$ possible monomials form a basis of $\cl(n)$ as a vector space, and the $2^{n+1}$ signed monomials $\pm \clif(b)$ form a multiplicative group $\smon(n)$, or just $\smon$ when the context is clear. It is easy to see that two signed monomials of degrees $a$ and $b$ commute if and only if $ab = 0 \pmod{2}$, and one could equivalently define Clifford algebras as commutative superalgebras with odd and even parts generated by the odd and even degree monomials, respectively.

The following facts are needed for Proposition~\ref{prop:codes}:

\begin{lemma}
\label{lem:clifford commutation}
The image of a code $L$ under $\clif$ is commutative if and only if for all $a, b \in L$,
$$(a \cdot b) + \wt(a)\wt(b) = 0 \pmod{2}.$$
\end{lemma}
\begin{proof}
Finally, consider $\clif(a) = \gamma_{a_1} \ldots \gamma_{a_r}$ and $\clif(b) = \gamma_{b_1}\ldots \gamma_{b_s}$, where $r = \wt(a)$ and $s = \wt(b)$. Note we can get from $\clif(a)\clif(b)$ to $\clif(b)\clif(a)$ in $\wt(a)\wt(b)$ transpositions, where we move, in order $\gamma_{b_1}, \cdots, \gamma_{b_s}$ through $\clif(a)$ to the left, picking up exactly $\wt(a)$ powers of $(-1)$. However, we've also overcounted once for each time $a$ and $b$ shared a generator $\gamma_i$, since $\gamma_i$ commutes with itself. Therefore, we have exactly $(a \cdot b) + \wt(a)\wt(b)$ powers of $(-1)$. The condition for commutativity is then that this quantity be even for all pairs $a$ and $b$, which is equivalent to the second defining condition for dashing codes.
\end{proof}

\begin{prop}
\label{prop:clifford weird group}
A code $L$ is a dashing code if and only if $L$ has the property that for a suitable sign function $s(v) \in \{\pm 1\}$ with $s(\overrightarrow{0}) = 1$, the set $\smon_L = \{s(v) \clif(v) \mid v \in L\}$ form a subgroup of $\smon$.
\end{prop}
\begin{proof}
Without loss of generality, say $\clif(v) = s(v) \prod_{i=1}^k \gamma_i$. Then
\begin{align*}
\clif(v)^2 & = (\gamma_1 \gamma_2 \cdots \gamma_{k}) (\gamma_1 \gamma_2 \cdots \Q_{k}) \\
& = (-1)^{(k-1)} (\gamma_2 \gamma_3 \cdots \gamma_{k}) (\gamma_1) (\gamma_1 \gamma_2 \cdots \gamma_{k}) \\
& = (-1)^{k(k-1)/2}.
\end{align*}
Suppose $s$ exists. Then, we must not have $(-1) \in \smon_L$ (since we already have $1 \in \smon_L$. Therefore, it is necessary to have the last quantity equal $1$, which happens exactly when $\wt(v) = 0$ or $1 \pmod{4}$ for all $v \in L$. Now, since $\clif(a)\clif(b)$ and $\clif(b)\clif(a)$ are equal up to sign, we must have them commute for all $a$ and $b$. By Lemma~\ref{lem:clifford commutation} this is the remainder of the requirement of $L$ being a dashing code. Thus, it is necessary for $L$ to be a dashing code for $s$ to exist.

If $L$ were a dashing code, then pick a basis $l_1, \ldots, l_k$ of $L$ and assign $s(l_i) = 1$ for all $i$. Note by the above equations $\clif(l_i)^2 = 1$ for all $i$. The linear independence of the $l_i$ is equivalent to saying that no group axioms are broken by this choice of $l_i$. Now, greedily define 
$$s(\prod_{i \in I} \clif(l_i)) = \prod_{i \in I} s(\prod(\clif(l_i)),$$
which is well-defined and closed under multiplication since the $l_i$ commute and square to $1$. 
\end{proof}

%Finally, we need to check that the $s(b)$ can be selected so that we have a group under multiplication. If $L$ is $k$-dimensional, take $k$ independent generators $l_1, \ldots, l_k$ of $L$. Set $s(l_i) = 1$ for all $i$. Since any bitstring in $l$ has weight $0$ or $1 \pmod 4$, there is at most one bitstring of odd weight. Since The multiplicative closure of the $l_i$ form a subgroup $M' \subset M$. If $M'$ contains exactly one sign $s(m)$ for each monomial $m$, we're done. Otherwise, $M'$ must simultaneously contain both $\pm m$ for each monomial $m$. 

%Note for any edge $(v,w)$ of color $i$ we have $\clif(v) = \pm \gamma_i \clif(w)$. We will discuss this sign further in Section~\ref{sec:dashing}.

\section{Lattices}
\label{app:lattices}

For a reference, see \cite{stanley:ec1} or any other standard treatment of posets.

An \emph{order ideal} of a poset $P$ is a subset of elements $S \subset P$ such that if $s \in S$ and $s > t$, then $t \in S$. There is a bijection between the order ideals of a poset $P$ and the \emph{antichains} (sets of pairwise incomparable elements) of $P$, where each order ideal $I$ is mapped to the set of maximal elements in $I$. We define $J(P)$ to be the poset of order ideals of $P$, ordered by inclusion. 

The \emph{least upper bound} $x \vee y$ of $x$ and $y$ in a poset $P$ is an element $z$ (which must be unique if it exists) such that $z \geq x$ and $z \geq y$ and for all $z' \geq x$ and $z' \geq y$, $z' \geq z$. We can similarly define the \emph{greatest lower bound} $x \wedge y$ when we reverse all the directions in the previous definition. A poset is a \emph{lattice} if every pair of elements has a least upper bound and a greatest lower bound. A lattice is \emph{distributive} if the following hold:
\begin{align*}
x \vee (y \wedge z) & = (x \vee y) \wedge (x \vee z) \\
x \wedge (y \vee z) & = (x \wedge y) \vee (x \wedge z)
\end{align*}

It is routine to check that $J(P)$ is a distributive lattice. Furthermore:

\begin{theorem} [Fundamental Theorem for Finite Distributive Lattices]
\label{thm:ftfdl}
The map $J$ is a bijection between finite posets and finite distributive lattices.
\end{theorem}

\bibliographystyle{plain}
\bibliography{adinkras}

\end{document}